\documentclass[11pt, a4paper]{article}

\usepackage{a4wide}

\usepackage{amssymb,amsmath,amsthm}
\usepackage{graphicx,enumerate}
\usepackage{ragged2e}
\usepackage{xcolor}
\usepackage[utf8]{inputenc}
\usepackage[T1]{fontenc}
\usepackage[english]{babel}
\usepackage{subcaption}
\usepackage{amssymb, amsmath, latexsym, amsthm, enumerate, amsfonts}
\usepackage[labelsep=period]{caption}
\usepackage{thmtools}
\usepackage{thm-restate}

\usepackage[font=small]{caption}
\usepackage{mwe}

\usepackage[pdftex, pdfstartview=FitH]{hyperref}

\usepackage{enumitem}

\usepackage{tikz}
\usetikzlibrary{decorations.pathreplacing,shapes}
\usetikzlibrary {positioning} 

\tikzstyle{v}=[circle,inner sep=1, minimum size =1 pt, line width = 1pt, draw=black, fill=black]

\numberwithin{equation}{section}
\newtheorem{theorem}{Theorem}[section]
 \newtheorem{thm}{Theorem}[section]
 \newtheorem{cor}[thm]{Corollary}

 \newtheorem{lem}[thm]{Lemma}
 
 \newtheorem{obs}[thm]{Observation}
 
 \theoremstyle{definition}
 \newtheorem{defn}[thm]{Definition}
 \theoremstyle{remark}

\newcommand{\dmg}{\mathop{}\!\mathrm{dmg}}
\newcommand{\KG}{\mathop{}\!\mathrm{KG}}
\newcommand{\dist}{\mathop{}\!\mathrm{d}}
\newcommand{\diam}{\mathop{}\!\mathrm{diam}}
\newcommand{\cart}{\, \Box \,}
\newcommand{\size}[1]{\left \vert #1 \right \vert}

\newcommand{\eps}{\varepsilon}

\newcommand{\Ex}{\mathbb{E}}
\newcommand{\Prob}{\mathbb{P}}

\newcommand{\cH}{{\mathcal H}}

\title{On the damage number of graphs}
\author{Valentin Gledel\,\thanks{Université Savoie Mont Blanc, CNRS UMR5127, LAMA, Chambéry, F-73000, France, e-mail: \href{mailto:valentin.gledel@univ-smb.fr}{valentin.gledel@univ-smb.fr}} \and William B. Kinnersley\,\thanks{Department of Mathematics and Applied Mathematical Sciences, University of Rhode Island, Kingston, Rhode Island, USA, e-mail: \href{mailto:billk@uri.edu}{billk@uri.edu}} \and Bal\'azs Patk\'os\,\thanks{HUN-REN Alfr\'ed R\'enyi Institute of Mathematics, Budapest, 1053, Re\'altanoda utca 13-15 and Department of Computer Science and Information Theory, Budapest University of
Technology and Economics, Budapest, Hungary, e-mail: \href{mailto:patkos@renyi.hu}{patkos@renyi.hu} } \and Milo\v s Stojakovi\' c\,\thanks{Department of Mathematics and Informatics, Faculty of Sciences, University of Novi Sad, 21000 Novi Sad, Serbia, e-mail: \href{mailto:milosst@dmi.uns.ac.rs}{milosst@dmi.uns.ac.rs}. Partly supported by the Science Fund of the Republic of Serbia, Grant \#7462: Graphs in Space and Time: Graph Embeddings for Machine Learning in Complex Dynamical Systems (TIGRA), and partly supported by the Ministry of Science, Technological Development and Innovation of the Republic of Serbia (grants 451-03-33/2026-03/200125 \& 451-03-34/2026-03/200125).}
}

\date{}

\begin{document}

\maketitle

\begin{abstract}
    We study a variant of Cops and Robbers in which the robber attempts to visit as many vertices of the graph as possible without being captured, while the cop aims to keep the robber confined to a small set of vertices. The \textit{damage number} of a graph $G$, introduced by Cox and Sanaei in 2019, is the maximum number of vertices the robber can visit in a game of Cops and Robbers on $G$. 
    
    In this paper, we determine damage numbers for several classes of graphs, including hypercubes, Hamming graphs, Johnson graphs, incidence graphs of projective planes, and Erd\H{o}s-Rényi random graph $G(n,p)$, for $p \gg \log^{2/5}(n) / n^{1/5}$.  We also show that the problem of determining the damage number of a graph is {\sf PSPACE}-complete.
\end{abstract}

\section{Introduction}

The game of Cops and Robbers is a pursuit-evasion game played on a graph, introduced independently by Quilliot~\cite{quilliot1978jeux} and by Nowakowski and Winkler~\cite{nowakowski1983vertex}. In the game, a set of cops and a robber occupy vertices of a graph $G$ and move alternately along edges, with the cops winning if one of them eventually lands on the robber’s vertex. The minimum number of cops required to guarantee capture is the cop number of the graph, a widely studied graph parameter connected to structural and algorithmic properties of graphs. A book by Bonato and Nowakowski~\cite{bonato2011game} covers this general topic in great detail. 

A recent variant, introduced by Cox and Sanaei~\cite{CS}, is the damage version of the game. Here the robber leaves permanent damage on every vertex it visits, and the objective of the cop (in this paper we consider games with one cop) is to minimize the total amount of damage incurred before capture. The corresponding parameter, called the damage number and denoted $\dmg(G)$, is the maximum number of vertices that robber can guarantee to  damage no matter the strategy of the cop. 

We note that the damage variant of Cops and Robbers was also studied with multiple cops by Carlson, Halloran and Reinhart~\cite{carlson2021damage}. Unlike the original game, in the damage variant it makes sense also to introduce multiple robbers, which was done by Carlson et al.~\cite{carlson2022multi}. These results were extended by Stojaković and Wolf~\cite{stojakovic2022multirobber}, where the authors also considered the game of two cops and two robbers. Recently, these and related problems were also investigated by Gagarova~\cite{gagarova2025exploring}.

Coming back to our setting with one cop and one robber, there has been research on the damage number of Cartesian product of graphs by Huggan, Messinger and Porter in~\cite{HMP, cor-huggan2024damage}. We start our investigations by considering the Cartesian product of cliques of arbitrary size, showing the following.

\begin{theorem}\label{thm:hamming}
If $G = K_{n_1} \cart K_{n_2} \cart \dots \cart K_{n_d}$, where $2 \le n_1 \le n_2 \le \dots \le n_d$, then 
\[\dmg(G) \le \begin{cases}
\displaystyle\frac{\size{V(G)}}{4} + \sum_{i=2}^{d} \prod_{j=3}^i n_j, \,\, &\text{if }n_1 = n_2 = 2\text{ and $n_d = 3$},\\[4ex]
\displaystyle\frac{\size{V(G)}}{n_d} + \sum_{i=0}^{d-2} \prod_{j=1}^i n_j, \,\, &\text{otherwise,}\end{cases}\]
and
\[\dmg(G) \ge \begin{cases}
\displaystyle\frac{\size{V(G)}}{4}, \,\, &\text{if }n_1 = n_2 = 2,\\[3ex]
\displaystyle\frac{\size{V(G)}}{n_d}, \,\, &\text{otherwise.} \end{cases}\]
\end{theorem}

The upper bound of the previous theorem can be improved when the values of the different $n_i$ are not all different. Indeed, in~\cite[Corollary 2.7]{HMP}, Huggan, Messinger and Porter proved that for any graphs $G$ and $H$, $\dmg(G\cart H) \leq \max\{\dmg(G)|H|, (\dmg(H)+1)|G|\}$. Applying this bound iteratively to the product of $K_n$ shows that for any $d$, we have $\dmg((K_n)^d) \leq n^{d-1}$ (by defining $G$ as $(K_n)^{d-1}$ and $H$ as $K_n$).

Using the lower bound of the previous theorem, this leads to the following corollary.

\begin{cor}
    For any $n \geq 3$, $\dmg((K_n)^d) = n^{d-1}$.
\end{cor}

Our next theorem determines the damage number of the missing case $n=2$ which is that of the hypercube $(K_2)^d=Q_d$. Note that the general result from~\cite{HMP} would yield an upper bound of $2^{d-1}$, which is twice the real value of $\dmg(Q_d)$ as proved in the following theorem.

\begin{theorem}\label{hypercube}
For all $d \ge 2$, we have $\dmg(Q_d) = 2^{d-2}$.
\end{theorem}

Although the vertex set of the hypercube is most often viewed as the set of 0,1-sequences of length $d$, via characteristic vectors the vertex set can be considered to be the power set of $[d]=\{1,2\dots,d\}$. Next, we address the damage number of graphs defined through intersection or containment properties of set families. For a set $S$, we denote by $\binom{S}{k}$ the family of all $k$-subsets of $S$.

The \textit{middle graph} $M_k$ is the bipartite graph with parts $\binom{[2k+1]}{k}$ and $\binom{[2k+1]}{k+1}$ and $F\in \binom{[2k+1]}{k},G\in \binom{[2k+1]}{k+1}$ are joined by an edge if and only if $F\subseteq G$.

The \textit{Johnson graph} $J(n,k)$ has vertex set $\binom{[n]}{k}$, and $F$ and $G$ are connected by an edge if and only if $|F\cap G|=k-1$.

We prove the following results.

\begin{theorem}\label{middle}
For any $k\ge 2$, we have $\dmg(M_k)=2\binom{2k-1}{k-1}$.
\end{theorem}

\begin{theorem}\label{johnson}
For any $n> k \ge 3$ and for $k=2, n\ge 4$, we have $\dmg(J(n,k))= \binom{n-2}{k-1}$.
\end{theorem}

Graph searching games have been studied in numerous settings on various random graph models, and the book by Bonato and Prałat~\cite{bonato-pralat-book} collects many of these results. In the paper where the damage number was first defined, Cox and Sanaei~\cite{CS} also introduce $D_{dmg}(G):=\dmg(G) /|V(G)|$ and investigate which values in the interval $[0,1]$ it can attain. They showed~\cite[Theorem 5.1]{CS} that for any $r\in [0, 1/2]$ and any $\varepsilon>0$ there exists a graph $G$ with $D_{dmg}(G)$ within $\varepsilon$ of $r$. In the part~(i) of the following theorem on the Erd\H os-Rényi random graph we settle this problem for all $r\in [0, 1]$.

\begin{theorem}\label{random}
\begin{itemize}
    \item[(i)] 
    For a constant $0<p<1$, we have $\dmg(G(n,p))=n-(1+o(1))pn$ w.h.p.
    \item[(ii)]
    For $\frac{\log^{2/5}n}{n^{1/5}} \ll p(n) \ll 1$, we have $\dmg(G(n,p))= n-\Theta(np)$ w.h.p.
\end{itemize}
\end{theorem}

Cox and Sanaei also established the upper bound $\dmg(G)\le |V(G)|-\Delta(G)-1$, achieved by having the cop defend the closed neighborhood of a vertex of maximum degree, and they asked \cite[Problem 6.2]{CS} for graphs for which this bound is tight. (The bound is achieved by having the cop defend the closed neighborhood of a vertex of maximum degree.) Observe that the Theorem \ref{random} (i) shows that the Erd\H os-R\'enyi random graph $G(n,p)$ with constant $p$ satisfies $\dmg(G)=n-(1+o(1))\Delta(G)$. The bound $|V(G)|-\Delta(G)-1$ cannot be tight for connected bipartite graphs (apart from stars) as the cop can defend the neighborhoods of two adjacent vertices $x$ and $y$ maximizing $d(x)+d(y)$. This shows that $\dmg(B)\le |V(B)|-\Delta_2(B)$ holds for all bipartite graphs $B$, where $\Delta_2(B)=\max\{d(x)+d(y):xy\in E(B)\}$. We show that the incidence bipartite graph of many (but not all!) projective planes almost achieves this bound. The number of vertices in the incidence graph $G$ of a projective plane of order $q$ is $2(2^2+q+1)$ and all degrees are $q+1$, so $\Delta_2(G)=2(q+1)$. The next theorem establishes that for many projective planes, we have $\dmg(G)=n-\Delta_2(G)-2$. (The definitions for all terms in the theorem below are provided in Section~\ref{s:results}.)

\begin{theorem}\label{thm:proj_plane2}
If $G$ is the incidence graph of a projective plane $\Pi$ of order $q$, then 
\[2q^2-2\sqrt{q}\le \dmg(G)\le 2q^2-2 .\]
Furthermore, 
\begin{itemize}
    \item 
    if $\Pi$ does not contain a Baer subplane, then $\dmg(G)=2q^2-2$;
    \item 
    if $\Pi$ contains a Baer subplane and is doubly transitive, then $\dmg(G) = 2q^2-2\sqrt{q}$. 
\end{itemize}
\end{theorem}

Finally, we address the complexity of the decision problem corresponding to the damage number. The input is a pair $(G,N)$ of a graph and a positive integer, and the answer is yes if and only $\dmg(G)\ge N$.

\begin{theorem}\label{complexity}
    The decision problem \textsc{Damage Number} is {\sf PSPACE}-complete.
\end{theorem}

The remainder of the paper is organized as follows: Section~\ref{s:prelim} contains some preliminary observations; in Section~\ref{s:results}, we prove our theorems on the damage number of certain graph classes; Section~\ref{s:complexity} contains the proof of Theorem~\ref{complexity}; and we end the paper with some concluding remarks and open problems in Section~\ref{s:conclusion}.

\section{Preliminaries}\label{s:prelim}

For a formal definition of the problem, let $G$ be a graph. To begin the game, the cop chooses an initial vertex to occupy, and then the robber chooses his initial vertex. Each round consists of a move of the cop followed by a move of the robber, where each of them has the opportunity to (but is not required to) move to a neighboring vertex. If the cop ever occupies the same vertex as the robber, the robber is captured and the game ends. The information about the locations and movements of both the cop and the robber is publicly available, i.e., this is a perfect information game.

Now we add the concept of damage to the setup. If a vertex $v$ is occupied by the robber \emph{at the end} of a round and the robber is not caught in the \emph{following} round, then $v$ becomes damaged and stays damaged for the rest of the game. The cop's goal is to minimize the number of damaged vertices, while the robber wants to damage as many vertices as possible. The \emph{damage number} of G, denoted $\text{dmg}(G)$, is the number of damaged vertices when both the cop and the robber play optimally.

We add a technical remark that we will use throughout the paper. Note that the possibility of the robber staying put at any point of the game can be excluded, as the cop can respond by also staying put until the robber moves, clearly allowing no extra damage by doing so. Since the robber's aim is not survival but rather damaging as many vertices as possible, this approach by the cop will not harm the cop's goal of minimizing the number of damaged vertices. Hence, w.l.o.g., from now on we require the robber to change position in every round.

We now gather a couple of lemmas that will be used as tools in the proofs in later sections. The first one is the standard Chernoff bound, see, e.g.,~\cite{AS}.

\begin{lem}[Chernoff bound] \label{lem:ch-bound}
    Let $X$ be a sum of independent indicator random variables, let $\mu:=\Ex[X]>0$, and let $0<f<1$. Then $\Prob[X\le (1-f)\mu] \le \exp\left( -\frac{f^2\mu}{2} \right)$. 
\end{lem}

The next observation is trivial as if the robber moves closer to a vertex $z$ that is already closer to his position than to the cop's, then after his move, he will be at least two steps closer to $z$ than the cop and therefore cannot be captured by the cop.

\begin{obs}\label{robbercloser}
    Suppose that immediately prior to a robber move, the robber occupies vertex $u$ and the cop occupies vertex $v$.  If for some vertex $z$ we have $d_G(u,z)<d_G(v,z)$, then the robber can damage $z$.
\end{obs}

For a graph $G$ and vertices $u,v\in V(G)$, we write $D_G(u<v)=\{w\in V(G): d_G(u,w)<d_G(v,w)\}$.  Observation \ref{robbercloser} leads to the following general lemma.

\begin{lem}\label{lem:generallower}
    For any graph $G$, we have $\dmg(G)\ge\min\{|D_G(u<v)|\}$, where the minimum is taken over all $u,v$ such that $\dist_G(u,v) \le 2$. Furthermore, if $G$ is triangle-free, then $\dmg(G)\ge\min\{|D_G(u<v)|\}$ with the minimum taken over all $u,v$ with $\dist_G(u,v) = 2$.
\end{lem}

\begin{proof}
    Suppose both players play optimally, and consider a point in the game after the robber has damaged $\dmg(G)$ vertices. Let $w$ be an undamaged vertex. As $\dmg(G)$ vertices have already been damaged, the cop must be able to stop the robber from damaging $w$. Suppose the robber starts moving towards $w$ on a shortest path from their current position. The only thing that can prevent the robber from reaching and damaging $w$ is that at some moment, the robber is at vertex $u$ and wants to move to vertex $u'$, but $u'$ is adjacent to the cop's current position $v$. But then $u$ and $v$ are at distance at most 2. By Observation~\ref{robbercloser}, the robber can damage any vertex in $D_G(u<v)$ and, as no further vertices can be damaged, all vertices in $D_G(u<v)$ must already be damaged; hence $\dmg(G) \ge |D_G(u<v)|$. Additionally, observe that that if $d_G(u,v)=1$, then $u,u'$, and $v$ form a triangle, so if $G$ is triangle-free, we need only consider the case where $\dist_G(u,v) = 2$.
\end{proof}

\smallskip

\noindent \textbf{Notation.} For any vertex $x$, $N(x)$ and $N[x]$ denote the open and closed neighborhoods of $x$, respectively. For any sequence $x_1,x_2,\dots,x_k$ of vertices, we write 
\[ N(\bar{x}_1,\bar{x}_2,\dots,\bar{x}_s,x_{s+1},\dots,x_k)=\bigcap_{i=1}^s\Big(V(G)\setminus N[x]\Big)\cap \bigcap_{i=s+1}^kN(x_i). \]

\section{Results for specific graph classes} \label{s:results}

In this section, we prove our results on $\dmg(G)$ for different graph classes. For the reader's convenience we will restate all our theorems before their proof. 
\subsection{Cartesian products of cliques}

\begingroup
\renewcommand{\thetheorem}{\ref{thm:hamming}} 
\begin{theorem}
If $G = K_{n_1} \cart K_{n_2} \cart \dots \cart K_{n_d}$, where $2 \le n_1 \le n_2 \le \dots \le n_d$, then 
\[\dmg(G) \le \begin{cases}
\displaystyle\frac{\size{V(G)}}{4} + \sum_{i=2}^{d} \prod_{j=3}^i n_j, \,\, &\text{if }n_1 = n_2 = 2\text{ and $n_d \le 3$};\\[4ex]
\displaystyle\frac{\size{V(G)}}{n_d} + \sum_{i=0}^{d-2} \prod_{j=1}^i n_j, \,\, &\text{otherwise,}\end{cases}\]
and
\[\dmg(G) \ge \begin{cases}
\displaystyle\frac{\size{V(G)}}{4}, \,\, &\text{if }n_1 = n_2 = 2;\\[3ex]
\displaystyle\frac{\size{V(G)}}{n_d}, \,\, &\text{otherwise.} \end{cases} \]
\end{theorem}
\addtocounter{theorem}{-1} 
\endgroup

\begin{proof}
We view the vertex set of $G$ as the set of ordered $d$-tuples in which the $i$th coordinate belongs to $\{0, \dots, n_i-1\}$.  We say that the cop and robber \textit{agree} in a coordinate if their positions have the same value in that coordinate; otherwise, they \textit{disagree}.

\medskip

For the upper bound, we first give a strategy for the cop to prevent the robber from damaging more than $\frac{\size{V(G)}}{n_d} + \sum_{i=0}^{d-2} \prod_{j=1}^i n_j$ vertices no matter what values the $n_i$ take; later, we will explain how the strategy can be improved when $n_1 = n_2 = 2$ and $n_d\le 3$.  The cop starts on the vertex whose coordinates are all 0.  Henceforth, on each turn, the cop finds the lowest-indexed coordinate in which the players disagree, and she changes that coordinate to agree with the robber.  Note that once the cop and robber agree in coordinates $1, \dots, k$, they will continue to agree after every cop turn: if the robber changes one of these coordinates, then the cop will do the same in response.  Additionally, once the players agree in all $d$ coordinates, the cop has captured the robber and the game ends.  For convenience, we may assume without loss of generality that the cop eventually captures the robber: for example, once the robber has damaged all vertices he can, we may suppose that he remains in place until the cop captures him.

\medskip

For $i \in \{0, \dots, d\}$, let $t_i$ denote the first round in which, after the cop's turn, the cop and robber agree in coordinates $1, \dots, i$ (where we view ``round 0'' as the players' initial placement).  For $i \in \{1, \dots, d\}$, let $S_i$ denote the set of previously-undamaged vertices that the robber damages from round $t_{i-1}$ up through round $t_i-1$.  Note that each vertex in $S_i$ must have the same values for coordinates $i, i+1, i+2, \dots, d$, since if the robber changed one of these coordinates in round $k$, then the cop would move to agree in coordinate $i$ in round $k+1$; thus we would have $t_{i} = k+1$, so the robber's new vertex would not be damaged until round $t_i$ (if indeed the robber survived long enough to damage it at all).  Thus, $\size{S_i} \le \prod_{j=1}^{i-1} n_j$.  It follows that 
\begin{align*}
\dmg(G) &= \sum_{i=1}^d \size{S_i} \le \sum_{i=1}^d \prod_{j=1}^{i-1} n_j  \\
&= \prod_{j=1}^{d-1} n_j + \sum_{i=1}^{d-1} \prod_{j=1}^{i-1} n_j= \prod_{j=1}^{d-1} n_j + \sum_{i=0}^{d-2} \prod_{j=1}^{i} n_j \\ 
&= \frac{\size{V(G)}}{n_d} + \sum_{i=0}^{d-2} \prod_{j=1}^{i} n_j,
\end{align*}
as claimed.  

\medskip

We now explain how the cop can do slightly better if $n_1 = n_2 = 2$ and $n_d \le 3$.  In this case, the cop plays similarly; however, instead of changing the lowest-indexed coordinate in which the two plays disagree, the cop restricts her attention to coordinates $3, \dots, d$ and changes the first of these coordinates in which the players disagree.  Once the cop agrees with the robber in coordinates $3, \dots, d$, the robber can never again change coordinates 1 or 2 without immediately being captured.  An analysis similar to the one used above now shows that 
\[\dmg(G) \le \frac{\size{V(G)}}{4} + \sum_{i=2}^{d} \prod_{j=3}^i n_j.\]

\medskip

For the lower bound, by Lemma \ref{lem:generallower} we have $\dmg(G)\ge\min\{|D_G(u<v)|\}$, where the minimum is taken over all vertices $u$ and $v$ with $\dist_G(u,v) \le 2$.  Let $u$ and $v$ be any two vertices satisfying $\dist_G(u,v) \le 2$; we will bound $\size{D_G(u<v)}$.

\medskip

Suppose first that $\dist_G(u,v) = 1$.  This means that $u$ and $v$ differ in exactly one coordinate, say coordinate $k$.  A given vertex $w$ is strictly closer to $u$ than to $v$ if and only if $w$ agrees with $u$ in coordinate $k$; the number of such vertices is clearly $\prod_{j \not = k} n_j$, which is at least $\prod_{j=1}^{d-1} n_j$, i.e. $\size{V(G)}/n_d$.  

\medskip

Now suppose instead that $\dist(u,v) = 2$.  Vertices $u$ and $v$ must differ in exactly two coordinates, say coordinates $k$ and $\ell$.  This time, a vertex $w$ is closer to $u$ than to $v$ if and only if either:
\begin{itemize}
    \item $w$ agrees with $u$ in both coordinate $k$ and coordinate $\ell$; or
    \item $w$ agrees with $u$ in exactly one of these coordinates and does not agree with $v$ in the other.
\end{itemize}
The number of vertices of the first type is $\prod_{j\not \in \{k,\ell\}} n_j$.
To count vertices of the second type, note that such a vertex either agrees with $u$ in coordinate $k$ but does not agree with either $u$ or $v$ in coordinate $\ell$, or agrees with $u$ in coordinate $\ell$ but does not agree with either $u$ or $v$ in coordinate $k$.  Thus, the number of vertices of the second type is given by
\[(n_{\ell}-2) \prod_{j\not \in \{k,\ell\}} n_j + (n_k-2) \prod_{j\not \in \{k,\ell\}} n_j.\]
This simplifies to $(n_{\ell}+n_k-4) \prod_{j\not \in \{k,\ell\}} n_j$. 
In total, we have 
\begin{align*}
\size{S} &\ge \prod_{j\not \in \{k,\ell\}} n_j + (n_{\ell}+n_k-4)\prod_{j\not \in \{k,\ell\}} n_j\\ 
        &= (n_k + n_{\ell} - 3) \prod_{j\not \in \{k,\ell\}} n_j\\
        &= (n_k + n_{\ell} - 3) \cdot \frac{\size{V(G)}}{n_k \cdot n_{\ell}}\\ &= \frac{\size{V(G)}}{n_{\ell}} + \frac{\size{V(G)}}{n_k} - 3\frac{\size{V(G)}}{n_k \cdot n_{\ell}}.
\end{align*}
Suppose without loss of generality that $n_{\ell} \ge n_k$.  If $n_{\ell} \ge 3$, then we have 
\[\frac{\size{V(G)}}{n_{\ell}} + \frac{\size{V(G)}}{n_k} - 3\frac{\size{V(G)}}{n_k \cdot n_{\ell}} \ge \frac{\size{V(G)}}{n_{\ell}} \ge \frac{\size{V(G)}}{n_{d}}.\]
Otherwise we must have $n_k = n_{\ell} = 2$, so 
\[\frac{\size{V(G)}}{n_{\ell}} + \frac{\size{V(G)}}{n_k} - 3\frac{\size{V(G)}}{n_k \cdot n_{\ell}} = \frac{\size{V(G)}}{4}. \qedhere \]
\end{proof}

In the case of the hypercube -- that is, the product of complete graphs of order 2 -- the lower bound in Theorem \ref{thm:hamming} gives the exact value of the damage number, as we next show.

\begingroup
\renewcommand{\thetheorem}{\ref{hypercube}} 
\begin{theorem}
For all $d \ge 2$, we have $\dmg(Q_d) = 2^{d-2}$.
\end{theorem}
\addtocounter{theorem}{-1} 
\endgroup

\begin{proof}
As usual, we view the vertex set of $Q_d$ as the set of ordered $d$-tuples in which each coordinate is either 0 or 1, with two vertices adjacent if and only if the corresponding $d$-tuples differ in exactly one coordinate.  We aim to prove the inequalities $\dmg(Q_d) \ge 2^{d-2}$ and $\dmg(Q_n) \le 2^{d-2}$.  

\medskip

The lower bound follows from Theorem \ref{thm:hamming}.  For the upper bound, we must give a strategy for the cop to ensure that the robber can damage no more than $2^{d-2}$ vertices.  On the cop's first turn, if the cop is at odd distance from the robber, then she moves toward the robber; otherwise, she remains on her current vertex.

\medskip

Throughout the remainder of the game, the cop plays as follows.  Suppose that on the preceding robber turn, the robber changed some coordinate $\ell$.  If, prior to the robber's move, the cop and robber had the same value for coordinate $\ell$, then the cop changes that coordinate in response.  Otherwise, the cop changes any coordinate in which she and the robber currently disagree.  Note that if both players agree in some coordinate at any point in the game, then the cop's strategy ensures that the players will agree in that coordinate after every subsequent cop turn for the remainder of the game.  Moreover, if the cop and robber disagree in only two coordinates and the robber changes one of these two, then the cop will change the second and capture the robber.  

\medskip

If the cop and robber initially disagree on one or fewer coordinates then the robber is captured before he can damage any vertices, so suppose otherwise.  If there are at least three coordinates in which the cop and robber disagree throughout the duration of the game, then it must be that the robber never changes those three coordinates and, consequently, cannot visit more than $2^{d-3}$ different vertices.  Finally, suppose that eventually, after each cop turn, the cop and robber disagree in two or fewer coodinates. 
Suppose by symmetry that the final two coordinates in which the cop and robber disagree (after a cop turn) are coordinates 1 and 2, and consider any point in the game when the players agree in all other coordinates.  The robber cannot yet have changed coordinates 1 or 2, and if he ever changes one, then the cop will change the other and capture him.  Hence, every damaged vertex must have the same values for those two coordinates, so the number of damaged vertices is at most $2^{d-2}$, as claimed.
\end{proof}

\subsection{Graphs defined via set families}
Recall that the middle graph $M_k$ is the bipartite graph with parts $\binom{[2k+1]}{k}$ and $\binom{[2k+1]}{k+1}$ and $F\in \binom{[2k+1]}{k},G\in \binom{[2k+1]}{k+1}$ are joined by an edge if and only if $F\subseteq G$.

\begingroup
\renewcommand{\thetheorem}{\ref{middle}} 
\begin{theorem}
For any $k\ge 2$, we have $\dmg(M_k)=2\binom{2k-1}{k-1}$.
\end{theorem}
\addtocounter{theorem}{-1} 
\endgroup

\begin{proof}
    For the lower bound, observe that $M_k$ is bipartite and so triangle-free. Also, for vertices $F,G\in V(M_k)$ at distance 2 with $\{i\}=F\setminus G,\{j\}=G\setminus F$, we have $D_G(F<G)=\{H\in V(M_k):i\in H,j\notin H\}$ and so $|D_G(F<G)|=2\binom{2k-1}{k-1}$. The statement now follows from Lemma \ref{lem:generallower}.

    \medskip
    
    The cop strategy that witnesses the upper bound is basically the same as in Theorem \ref{hypercube}, so we just sketch the proof. We write $C_i$, $R_i$ to denote the sets corresponding to the position of the cop and the robber, respectively after the $i$th move of the cop. If the two starting sets (vertices) are of the same size, then the cop stays put for their first move. Otherwise, the cop always gets closer to the robber as follows: if the robber in their previous move removed a vertex $x\in C_i\cap R_i$ from $R_i$ or added a vertex $y \in [2k+1]\setminus (R_i\cup C_i)$ to $R_i$, then the cop removes $x$ or adds $y$, respectively. If that was not the case, i.e. the robber added / removed a vertex $x\in R_i\triangle C_i$, then the cop adds / removes a vertex $x'\in R_i\triangle C_i$ with $x'\neq x$ to make $d_{M_k}(R_{i+1},C_{i+1})=d_{M_k}(R_i,C_i)-2$. Observe that with this strategy, we have $R_{i+1}\triangle C_{i+1}\subseteq R_i\triangle C_i$ and due to the first move of the cop, we always have $d_{M_k}(R_i,C_i)=|R_i \triangle C_i|$ even. Also, the elements of $R_i\triangle C_i$ had never changed before the $i$th move, so the number of damaged vertices up to this point is at most $2\binom{2k+1-|R_i\triangle C_i|}{k-\frac{1}{2}|R_i\triangle C_i|}$. Clearly, $|R_i\triangle C_i|\ge 2$ as otherwise $|R_i\triangle C_i|=0$ and the cop captures the robber, so $\dmg(M_k)\le 2\binom{2k-2}{k-1}$, as claimed. 
\end{proof}

Recall that the Johnson graph $J(n,k)$ has vertex set $\binom{[n]}{k}$ with vertices $F$ and $G$ connected by an edge if and only if $|F\cap G|=k-1$.

\begingroup
\renewcommand{\thetheorem}{\ref{johnson}} 
\begin{theorem}
For any $n> k \ge 3$ and for $k=2, n\ge 5$, we have $\dmg(J(n,k))= \binom{n-2}{k-1}$. 
\end{theorem}
\addtocounter{theorem}{-1} 
\endgroup

\begin{proof}
    We start by describing and analyzing the cop's strategy. The cop chooses her starting vertex arbitrarily and remains in place on her first turn. We explain how the cop moves for the remainder of the game. Suppose that after the cop's $i$th move, the $k$-sets occupied by the robber and the cop are $R_i$ and $C_i$, respectively. Let the robber's next move be $R_{i+1}=R_i \setminus \{x\}\cup \{y\}$. We distinguish several cases:
    
    \begin{enumerate}
        \item 
        If $x\in R_i\cap C_i$ and $y \not \in C_i$, then we let $C_{i+1}=C_i\setminus \{x\}\cup \{y\}$.
        \item 
        If $x\in R_i\setminus C_i$ and $y \not \in C_i$, then we pick an arbitrary $x'\in C_i\setminus R_i$ and let $C_{i+1}=C_i\setminus \{x'\}\cup \{y\}$.
        \item 
        If $y\in C_i$, then we choose $C_{i+1}$ such that $|R_{i+1}\cap C_{i+1}| \ge |R_i\cap C_i|+1$. This is possible because if $x\in R_i\cap C_i$, then we let $C_{i+1}=C_i\setminus \{x\}\cup \{y'\}$, where $y'\in R_i\setminus C_i$ is arbitrary; otherwise, picking arbitrary $x'\in C_i\setminus R_i$ with $x' \not = y$ and $y'\in R_i\setminus C_i$ with $y' \not = x$, we let $C_{i+1}=C_i\setminus \{x'\}\cup \{y'\}$.
    \end{enumerate}
    The following observation is immediate from the cop's strategy.
    \begin{obs}\label{nested}
           $R_{j+1}\setminus C_{j+1}\subseteq R_j\setminus C_j$ and $C_{j+1}\setminus R_{j+1}\subseteq C_j\setminus R_j$.
    \end{obs}
    Let $j^*$ denote the first moment in the game when $R_{j^*}\setminus C_{j^*}$ (and thus $C_{j^*}\setminus R_{j^*}$) reach their minimum size.  Let us write $R:=R_{j^*}\setminus C_{j^*}$ and $C:=C_{j^*}\setminus R_{j^*}$.
    Because of Observation \ref{nested}, all sets $R_j$ damaged by the robber contain $R$ and are disjoint with $C$. If $|R|=|C|=0$, then the robber is captured and $R_{j^*}$ is not damaged, so we can assume $|R|=|C|\ge 1$. Then the number of damaged $k$-sets is at most $\binom{n-|R|-|C|}{k-|R|}\le \binom{n-2}{k-1}$, as claimed. 

    \medskip

    For the lower bound, we will use Lemma \ref{lem:generallower}. In $J(n,k)$, we have $d_{J(n,k)}(H,H')=|H\setminus H'|=|H'\setminus H|$, and so if $H,H'$ are at distance 1, then $D_{J(n,k)}(H<H')$ consists of those sets that contain the single element of $H\setminus H'$ and do not contain the single element of $H'\setminus H$, so $|D_{J(n,k)}(H<H')|=\binom{n-2}{k-1}$. If $H,H'$ are at distance two, then $D_{J(n,k)}(H<H')$ consists of those sets that contain more elements from $H\setminus H'$ than from $H'\setminus H$. Therefore $|D_{J(n,k)}(H<H')|=2\binom{n-4}{k-1}+\binom{n-4}{k-2}+2\binom{n-4}{k-3}$, with the three terms counting those sets containing $a$ elements of $H \setminus H'$ and $b$ elements of $H' \setminus H$ for $(a,b)$ equal to $(1,0)$, $(2,0)$, and $(2,1)$, respectively.  We want to argue that $|D_{J(n,k)}(H<H')| \ge \binom{n-2}{k-1}$.  The inequality
    \[
    \binom{n-2}{k-1}\le 2\binom{n-4}{k-1}+\binom{n-4}{k-2}+2\binom{n-4}{k-3}
    \]
    is, after dividing by $\binom{n-4}{k-3}$, equivalent to
    \[
    \frac{(n-2)(n-3)}{(k-1)(k-2)} \le \frac{2(n-k-1)(n-k-2)}{(k-1)(k-2)} +\frac{(k-1)(n-k-1)}{(k-1)(k-2)} +\frac{2(k-1)(k-2)}{(k-1)(k-2)}. 
    \]
    After rearranging, we obtain
    \[
    0\le \frac{n^2-(3k+2)n+3k^2+3}{(k-1)(k-2)}.
    \]
    The denominator is positive for all $k\ge 3$. If $k\ge 4$, then the discriminant is negative and so the inequality holds. If $k=3$, then the numerator is $n^2-11n+30=(n-5)(n-6)$, which is non-negative for all integer $n$. In the case of $k=2$, the value of $|D_{J(n,2)}(H<H')|$ is $n-2$ when $|H\setminus H'| = 1$ and $2n-7$ when $|H\setminus H'| = 2$, respectively; in either case, $|D_{J(n,2)}(H<H')| \ge \binom{n-2}{k-1} = n-2$ when $n\ge 5$.
    
    Therefore, the minimum in Lemma \ref{lem:generallower} is $\binom{n-2}{k-1}$ which finishes the proof of the lower bound.
\end{proof}

\subsection{Incidence graphs of finite projective planes}
Recall that a \textit{finite projective plane} is a collection of \textit{points} and \textit{lines} such that:
\begin{itemize}
\item [\textbf{(1)}] any two distinct lines contain exactly one common point; 
\item [\textbf{(2)}] any two distinct points belong to exactly one common line; and
\item [\textbf{(3)}] there exist four points, of which no line contains more than two.
\end{itemize}
It is a well-known consequence of the axioms that in any projective plane, there exists some $q$ such that every line contains $q+1$ points and every point lies on $q+1$ lines; $q$ is referred to as the \textit{order} of the projective plane.  It is well-known that finite projective planes exist for any order $q$ that is a power of a prime.

The \textit{incidence graph} of a projective plane $\Pi$ contains one vertex corresponding to each point in $\Pi$, one vertex corresponding to each line, and edges joining each line to all points on that line.


\begin{defn}
Let $G$ be the incidence graph of a projective plane of order $q$, and let $S$ be a subset of $V(G)$.  We say that $S$ is \textit{closed} if for any two vertices $u$ and $v$ in $S$, the common neighbor of $u$ and $v$ in $G$ (if it exists) also belongs to $S$.
\end{defn}

Closed vertex sets are, in a sense, ``easy'' for a cop to protect, as we next show.

\begin{lem}\label{lem:proj_plane_helper}
Let $G$ be the incidence graph of a projective plane of order $q$, let $S$ be a set of vertices in $G$, let $S_P$ (respectively $S_L$) be the subset of $S$ containing those vertices that correspond to points (resp. lines), and let $x \in S_P$ be the starting vertex of the cop.  If $S$ is closed and the robber begins the game on some vertex $y$ such that $y \not \in S$ and every neighbor of $y$ in $S_L$ is also a neighbor of $x$, then the cop can prevent the robber from damaging any vertex in $S$.
\end{lem}
\begin{proof}
The cop begins the game on $x$.  Suppose the robber begins on some vertex $y$ such that $y \not \in S$ and every neighbor of $y$ in $S_L$ is also a neighbor of $x$.  We may assume that $y \not \in N(x)$, since otherwise the cop could capture the robber on her first turn.  

\medskip
Throughout the game, the cop will maintain the following invariants after each cop turn: 
\begin{itemize}
    \item[(i)] the cop occupies a vertex of $S$;
    \item[(ii)] the cop and robber occupy vertices in the same partite set of $G$;
    \item[(iii)] the robber does not occupy a vertex in $S$; and
    \item[(iv)] all of the robber's neighbors in $S$ are also neighbors of the cop.
\end{itemize}
By the assumption that $y \not \in S$, invariant (iii) holds prior to the cop's first move, so we need only concern ourselves with invariants (i), (ii), and (iv).  If $y$ corresponds to a point, then the cop remains on his current vertex.  Invariants (i) and (ii) clearly hold; moreover, every neighbor of $y$ in $S$ must belong to $S_L$, and by assumption every neighbor of $y$ in $S_L$ is also a neighbor of $x$, so invariant (iv) also holds.  Suppose instead that $y$ corresponds to a line.  If $y$ has no neighbors in $S_P$, then the cop moves to any vertex in $S_L$, and again the invariants clearly hold.  Suppose otherwise.  Note that $y$ can only have one neighbor in $S_P$: if $y$ had two such neighbors, say $z$ and $z'$, then the unique common neighbor of $z$ and $z'$ in $G$ must be $y$; by closure of $S$, the unique common neighbor of any two points in $S$ must lie in $S$, hence $y \in S$, contradicting an earlier assumption.  Thus, let $z$ denote the unique neighbor of $y$ in $S_P$; the cop moves to the unique common neighbor of $x$ and $z$ (which, by closure of $S$, lies in $S_L$).  Once again, the invariants clearly hold.

\medskip

Thus, the invariants hold after the first cop turn.  Suppose that the invariants hold after the $k$th cop turn; we will show how the cop can ensure that they hold after the $(k+1)$st.  

\medskip

Consider the state of the game after the robber's $k$th turn.  If the robber did not move on his $k$th turn, then the cop also remains at his current vertex, which clearly maintains all four invariants.  Suppose instead that the robber did move to a new vertex, say $z$.  If $z \in S$ then, because invariant (iv) held after the cop's $k$th turn, the cop must be adjacent to $z$ and can thus capture the robber, ending the game.  Hence we may suppose that $z \not \in S$.  By symmetry, we may assume that $z \in S_P$.  Because invariants (i) and (ii) held after the $k$th cop turn, the cop must currently occupy some vertex in $S_L$, say $w$.  If $z$ has a neighbor $z'$ in $S_L$, then the cop moves to the unique common neighbor of $w$ and $z'$; it is easily checked that this maintains all four invariants.  If instead $z$ has no neighbor in $S_L$, then the cop moves to any vertex of $S_P$; again, this maintains the invariants.

\medskip

Since the invariants hold after every cop turn, the robber can never reach a vertex in $S$ without immediately being captured; hence, he can never damage any vertex in $S$.
\end{proof}

The following lemma will help us bound the size of a closed vertex set. 

\begin{lem}\label{lem:proj_plane_closed}
Let $G$ be the incidence graph of a projective plane $\Pi$ of order $q$, let $S$ be a set of vertices in $G$, and let $S_P$ (respectively $S_L$) be the subset of $S$ containing those vertices that correspond to points (resp. lines) in $\Pi$.  If $S$ is closed and $\size{S_P} \ge q+3$, then $S$ induces the incidence graph of a subplane of $\Pi$.
\end{lem}
\begin{proof}
Suppose $S$ is closed and $\size{S_P} \ge q+3$.  Since $S$ is closed, $S$ satisfies axioms (1) and (2) of a projective plane, so we need only check axiom (3).  

\medskip

If no three points in $S$ are collinear, then any four points of $S$ satisfy the axiom.  Otherwise, let $x$, $y$, and $z$ be collinear point vertices in $S$, and let $w$ be their common neighbor.  Since no line contains more than $q+1$ points, and since $\size{S_P} \ge q+3$, there must be point vertices $u$ and $v$ in $S$ that are not adjacent to $w$.  

\medskip

Consider the set $\{u, v, x, y\}$.  By choice of $u$ and $v$, neither can be collinear with both $x$ and $y$.  If neither $x$ nor $y$ lies on a common line together with both $u$ and $v$, then the set $\{u,v,x,y\}$ satisfies the axiom, as desired.  Suppose instead that $u$, $v$, and $x$ are collinear (the case where $u, v$, and $y$ are collinear is similar).  Consider the set $\{u, v, y, z\}$.  Once again, neither $u$ nor $v$ can share a common line with both $y$ and $z$.  It also cannot be that vertices $u$, $v$, and $y$ are collinear: by assumption the common neighbor of $u$ and $v$ is also adjacent to $x$, so if it were also adjacent to $y$, then it would have to be $w$, the common neighbor of $x$ and $y$; however, by choice of $u$ and $v$, this is not the case.  Similarly, $u$, $v$, and $z$ cannot be collinear.  Hence the set $\{u,v,y,z\}$ satisfies axiom (3) of a projective plane, which completes the proof.
\end{proof}

Now we are ready to prove Theorem \ref{thm:proj_plane2}, but we need to define all notions it uses. Let $\Pi$ be a projective plane of order $q$. A \textit{Baer subplane} $\Pi'$ of $\Pi$ is a subplane of order $\sqrt{q}$. Therefore, $\Pi$ does not contain any Baer subplanes if $q$ is not a square and so the first point of the furthermore part of Theorem \ref{thm:proj_plane2} applies to many of the known projective planes. A projective plane $\Pi$ is \textit{doubly transitive} if for any two pairs of points $x,y$ and $x',y'$ in $\Pi$ there exists an automorphism $\phi$ of $\Pi$ such that $\phi(x)=x',\phi(y)=y'$ and the same hold for any two pairs $\ell_1,\ell_2$ and $\ell'_1,\ell'_2$ of lines. The most well-known projective planes are $PG(2,q)$ where $q$ is a prime power, the points of $PG(2,q)$ are the 1-dimensional subspaces of a 3-dimensional vector space $V$ over the finite field $\mathbb{F}_q$ of $q$ elements, and lines are the 2-dimensional subpaces of $V$. It is clear that $PG(2,q)$ is doubly transitive and if $q$ is square, then it contains a Baer subplane, so the second point of the furthermore part of Theorem \ref{thm:proj_plane2} applies to $PG(2,q)$.

\begingroup
\renewcommand{\thetheorem}{\ref{thm:proj_plane2}} 
\begin{theorem}
If $G$ is the incidence graph of a projective plane $\Pi$ of order $q$, then 
\[2q^2-2\sqrt{q}\le \dmg(G)\le 2q^2-2 .\]
Furthermore, 
\begin{itemize}
    \item 
    if $\Pi$ does not contain a Baer subplane, then $\dmg(G)=2q^2-2$,
    \item 
    if $\Pi$ contains a Baer subplane and is doubly transitive, then $\dmg(G) = 2q^2-2\sqrt{q}$. 
\end{itemize}
\end{theorem}
\addtocounter{theorem}{-1} 
\endgroup

\begin{proof}
Let $V_P$ and $V_L$ denote the vertices of $G$ corresponding to the points and lines of $\Pi$, respectively.  

\medskip

We begin with the lower bounds.  Initially, the robber begins on any vertex not adjacent to the cop.  Consider a later point in the game, on which it is the robber's turn and the robber has damaged $\dmg(G)$ vertices.  Let $S$ denote the set of vertices in $G$ that the robber cannot reach, after some number of turns, without being captured by the cop.  Since the robber can safely reach (and therefore damage) every vertex not in $S$, all undamaged vertices must belong to $S$; hence $\dmg(G) \ge \size{V(G)} - \size{S}$. 

\medskip

Thus motivated, we seek to bound $\size{S}$.  We first claim that $S$ must be closed.  Suppose to the contrary that some vertex $v$ outside $S$ has two neighbors $x$ and $y$ in $S$.  By definition of $S$, the robber can eventually reach $v$ and, on his subsequent turn, move to either $x$ or $y$.  The cop cannot simultaneously defend both $x$ and $y$: the unique common neighbor of $x$ and $y$ is $v$, and since the robber has reached $v$ without being captured, clearly the cop does not occupy $v$.  Hence no matter how the cop responds, the robber can reach at least one vertex in $S$, which is a contradiction; it follows that $S$ must be closed.

\medskip

By Lemma \ref{lem:proj_plane_closed}, either $S$ contains at most $q+2$ point vertices and (by symmetry) at most $q+2$ line vertices, or it induces a subplane $\Pi'$ of $\Pi$.  In the former case we have $\size{S} \le 2q+4$ and hence
\[\dmg(G) \ge \size{V(G)} - \size{S} \ge 2q^2+2q+2 - (2q+4) = 2q^2-2.\]
In the latter case, let $m$ be the order of $\Pi'$; it is well-known (see, e.g., \cite[Theorem 1.12]{KissSzonyi}) that $m^2 = q$ if $\Pi'$ is a Baer subplane, and $m^2+m \le q$ otherwise.  Thus, if $\Pi$ contains a Baer plane, we have $\size{S} \le 2(m^2+m+1) = 2q+2\sqrt{q}+2$, and so
\[\dmg(G) \ge \size{V(G)} - \size{S} \ge 2q^2+2q+2 - (2q+2\sqrt{q}+2) = 2q^2-2\sqrt{q}.\]
If not, then
\[\dmg(G) \ge \size{V(G)} - \size{S} \ge 2q^2+2q+2 - (2q+2) = 2q^2;\]
however, note that this is stronger than the bound from the first case, so we can only conclude that $\dmg(G) \ge 2q^2-2$.  This completes the proof of the lower bounds.

\medskip

For the upper bounds, we use Lemma \ref{lem:proj_plane_helper} to show that the cop can prevent the robber from damaging many vertices.  The cop selects any vertex $x$ in $V_P$ and begins the game on $x$.  We may assume that the robber does not begin on a vertex adjacent to the cop.  Let $y$ be any vertex of $V_L$ that is adjacent neither to $x$ nor to the robber.   Applying Lemma \ref{lem:proj_plane_helper} with $S = N[x] \cup N[y]$ shows that the cop can prevent the robber from damaging any vertex in $S$.  Hence, 
\[\dmg(G) \le \size{V(G)} - \size{S} = 2(q^2+q+1) - (2q+4) = 2q^2-2.\]

If $\Pi$ contains a Baer subplane and is doubly transitive, then the cop can do better.  
Let $z$ be an arbitrary vertex that the cop picks as start position and let $y$ be the start position of the robber. We can assume that $y\notin N[z]$ as then the cop will immediately capture the robber. Let $x$ be an arbitrary vertex of $N[z]$ in the same part of $G$ as $y$. In particular, $x=z$ if $z$ and $y$ are in the same part. Let $S'$ be the vertices of $G$ corresponding to a Baer subplane of $\Pi$. Let $x'$ be an arbitrary vertex in $S'$ in the same part of $G$ as $y$. Let $u'\in S'$ be a neighbor of $x'$. As $u'$ has $q+1$ neighbors in $G$ and only $\sqrt{q}+1$ of them are in $S'$, there is a neighbor $y'$ of $u'$ with $y'\notin S'$. As $\Pi$ is doubly transitive, there exists an automorphism $\phi$ of $\Pi$ such that $\phi(x')=x,\phi(y')=y$. After the first move of the cop, she will be at $x$ and we can apply Lemma \ref{lem:proj_plane_helper} to $x$, $y$ and $S=\phi(S')$. Now,
\[\dmg(G) \ge \size{V(G)} - \size{S} \ge 2q^2+2q+2 - (2q+2\sqrt{q}+2) = 2q^2-2\sqrt{q}. \qedhere\]
\end{proof}

\subsection{The Erd\H os-R\'{e}nyi random graph $G(n,p)$}
Formally, $G(n,p)$ is the probability space of all labeled graphs on $n$ vertices with $\mathbb{P}(G(n,p)=G)=p^{|E(G)|}(1-p)^{\binom{n}{2}-|E(G)|}$. Equivalently, for any pair $x,y$ of vertices, $xy\in E(G(n,p))$ with probability $p$ independently of all other pairs. We say that a sequence $\mathcal{E}_n$ of events hold \textit{with high probability} (\textit{w.h.p.}~for short) if $\mathbb{P}(\mathcal{E}_n)$ tends to 1 as $n$ tends to infinity.

We begin with an auxiliary lemma.  In the lemma below, as well as in the proof of Theorem \ref{random}, we use the notation $N(\bar{x_1}, \bar{x_2}, \dots, \bar{x_k}, y)$ to denote the set of vertices that are nonadjacent to all of $x_1, x_2, \dots, x_k$ but are adjacent to $y$ (not including the vertices $x_1, \dots, x_k, y$ themselves).

\begin{lem} \label{lem:chernoff}
Let $p=p(n)<1$, $f=f(n)<1$ and $k=k(n)$ a positive integer be such that $npf^2(1-p)^{k} \gg k \log \frac{en}{k}$, then w.h.p.~for all sequences of different vertices $x_1, x_2, \dots, x_i, x_{i+1}$ of $G(n,p)$, with $1\leq i\leq k$, it holds that 
\[|N(\bar{x}_1, \bar{x}_2,\dots, \bar{x}_i, x_{i+1})| \geq (1-f)p(1-p)^{i}n. \]
\end{lem}

\begin{proof}
    For fixed $x_1, x_2, \dots, x_i, x_{i+1}$, we have $\mu:=\Ex\big[|N(\bar{x}_1, \bar{x}_2,\dots, \bar{x}_i, x_{i+1})|\big] = (n-i-1)p(1-p)^{i} = \Theta(np(1-p)^{i})$. Therefore, by the Chernoff bound in Lemma~\ref{lem:ch-bound}, we get
    \[ \Prob \big[|N(\bar{x}_1, \bar{x}_2,\dots, \bar{x}_i, x_{i+1})| < (1-f) \mu\big] \leq \exp\left(-f^2\mu/2\right) = \exp\left(-\Theta(npf^2(1-p)^i)\right). \]
    
    By the union bound, the probability that this event does not happen for any choice of $x_1, x_2, \dots, x_i, x_{i+1}$, with $1\leq i\leq k$, is upper bounded by
    \[ 
    \begin{split} \sum_{i=1} ^{k} n\binom{n}{i} \exp\left(-\Theta(npf^2(1-p)^i)\right) &\leq \sum_{i=1} ^{k} \exp\left( i\log \frac{en}{i} - \Theta(npf^2(1-p)^i)\right) \\
    &\leq k \exp\left( k\log \frac{en}{k} - \Theta(npf^2(1-p)^k)\right).
    \end{split}
    \]
    Knowing that $\log k=o\big(\max\{k,\log n\}\big)$, the condition of the lemma now implies that the last expression tends to zero when $n$ tends to infinity.
\end{proof}

\begingroup
\renewcommand{\thetheorem}{\ref{random}} 
\begin{theorem}
\begin{itemize}
    \item[(i)] 
    For a constant $0<p<1$, we have $\dmg(G(n,p))=n-(1+o(1))pn$ w.h.p.
    \item[(ii)]
    For $\frac{\log^{2/5}n}{n^{1/5}} \ll p(n) \ll 1$, we have $\dmg(G(n,p))= n-\Theta(np)$ w.h.p.
\end{itemize}
\end{theorem}
\addtocounter{theorem}{-1} 
\endgroup

\begin{proof}
In both cases, the maximum degree of $G(n,p)$ is $(1+o(1))pn$ w.h.p. Hence, the cop can protect the closed neighborhood of a maximum degree vertex, implying $\dmg(G(n,p))\le n-(1+o(1))pn$ w.h.p. This proves the upper bounds in both cases.

As for the robber's side, assume that he uses his optimal strategy and consider a point in the game after he has damaged all the vertices that he can. Denote by $x_1$ the position of the robber, and by $y_1$ the position of the cop. Based on our assumption, all vertices in $N(x_1, \bar{y}_1)$ are damaged, otherwise the robber could move to an undamaged vertex outside of the cop's neighborhood. Then the robber moves to a vertex $x_2$ from this neighborhood, and denote by $y_2$ the vertex the cop moves to next. Similarly, all vertices in $N(x_2, \bar{x}_1, \bar{y}_2)$ are damaged, and the robber can move to one of them next, denote it by $x_3$. As long as these neighborhoods are non-empty, the robber can continue to move further along the path $x_1,x_2, \dots, x_i$; along the way, each $N(x_i, \bar{x}_1,\dots,\bar{x}_{i-1}, \bar{y}_i)$ must be completely damaged. Furthermore, all these neighborhoods are clearly pairwise disjoint. Let $\ell$ be the number of steps the robber can continue this process.
Then $\dmg(G(n,p)) \geq \sum_{i=1}^\ell \left| N(x_i, \bar{x}_1,\dots,\bar{x}_{i-1}, \bar{y}_i)\right|$.

\paragraph{(i)} 
For fixed $p$ let $k$ be an arbitrary fixed  constant and let $f=n^{-1/3}$, so that the conditions of Lemma~\ref{lem:chernoff} are satisfied. Note that since $k$ is constant, we have $(1-f)p(1-p)^in>1$ for all $i\le k$, so $\ell \ge k$ w.h.p.
\[
\sum_{i=1}^k\left| N(x_i, \bar{x}_1,\dots,\bar{x}_{i-1}, \bar{y}_i)\right|\ge \sum_{i=1}^k (1-f)p(1-p)^{i}n=(1-f)pn\frac{(1-p)-(1-p)^{k+1}}{p}.
\]
As $k$ grows, $(1-p)^{k+1}$ tends to zero, and thus the left hand side is at least $n-(1+o(1))pn$ w.h.p.

\paragraph{(ii)} Let $k=\frac{\log\frac{1}{p}}{p}$ and let $f=p$. Note that $(1-p)^k = p(1+o(1))$. Now, the condition of Lemma~\ref{lem:chernoff} is 
\[np^4(1+o(1)) \gg \frac{\log\frac{1}{p}}{p} \log \left(\frac{enp}{\log\frac{1}{p
}} \right). \]
Using our assumption on $p$, we have $np^5(1+o(1))\gg \log^2 n$, and $\log\frac{1}{p} \log \left(\frac{enp}{\log\frac{1}{p}}\right) = O(\log^2 n)$, so the condition of the lemma is satisfied.
Now for all $i\le k$ we have 
\[(1-f)p(1-p)^in=p(1-p)^{i+1}n > p(1-p)^{k+1}n=(1-o(1))p^2n>1,\]
so $\ell \ge k$ w.h.p.

Hence, having in mind that $(1-p)^k = p(1+o(1))$, we have
\[\begin{split}
\sum_{i=1}^k\left| N(x_i, \bar{x}_1,\dots,\bar{x}_{i-1}, \bar{y}_i)\right| & \ge \sum_{i=1}^k (1-f)p(1-p)^{i}n \\ &=pn(1-p)^2\sum_{i=0}^{k-1}(1-p)^i\\ &=n(1-p)^2\big(1-(1-p)^{k}\big)\\ &=n(1-p)^3(1-o(1))\\ &=n-3np+O(np^2),
\end{split} \]
implying the desired lower bound.
\end{proof}

\section{\textsc{Damage Number} is {\sf PSPACE}-complete} \label{s:complexity}

In this section, we explore the computational complexity of computing the damage number of a graph.  More precisely, we consider the following decision problem:\\

\noindent \textsc{Damage Number}: Given a graph $G$ and positive integer $N$, is $\dmg(G)$ at most $N$?\\

Our main result in this section is that \textsc{Damage Number} is {\sf PSPACE}-complete.  It is straightforward to show that the problem belongs to {\sf PSPACE}. Indeed, there are $O(n^2)$ possible configurations for the player's positions, and under optimal play, if a configuration repeats before the end of the game, then at least one more vertex must have been damaged between any two repetitions. Therefore the length of the game is at most $O(n^3)$ and so the game is in {\sf PSPACE} (e.g., because it can be solved in polynomial space using a standard recursive backtracking algorithm).

The rest of this section is dedicated to proving that {\sc Damage Number} is {\sf PSPACE}-hard.  We prove this by reducing from the standard problem {\sc 3-QBF}:\\

\noindent \textsc{3-QBF}: Given variables $X_1, \dots, X_n$ and a boolean formula $\varphi(X_1, \dots, X_n)$ in 3-CNF form, is the quantified expression 
\[\exists X_1 \, \forall X_2 \, \exists X_3 \, \dots \, \varphi(X_1, \dots, X_n)\]
true?\\

Recall that \textit{3-CNF form} means that the formula $\varphi$ consists of the conjunction of multiple \textit{clauses}, each of which is the disjunction of at most three \textit{literals} (i.e., $x_i$ or $\overline{x}_i$).  \textsc{3-QBF} is well-known to be {\sf PSPACE}-complete~\cite{GareyJohnson}.  
\textsc{3-QBF} is often viewed as a game between two players, Satisfier and Falsifier.  The players take turns: first Falsifier chooses a value for $X_1$, then Satisfier chooses a value for $X_2$, then Falsifier chooses a value for $X_3$, and so on.  Once values have been chosen for all variables, Satisfier wins the game if $\varphi(X_1, \dots, X_n)$ is true, while Falsifier wins if $\varphi(X_1, \dots, X_n)$ is false.  From this viewpoint, \textsc{3-QBF} is the problem of deciding whether or not Satisfier wins the game.

\subsection{Building the reduction}
\label{sub:reduc}
Given an instance of {\sc 3-QBF} with $n$ variables $X_1,\dots,X_n$ and $m$ clauses $C_1,\dots,C_m$, we build an instance $(G=(V,E),N)$ of {\sc Damage Number}.
We let $N = 2nm+4$ and the graph $G$ is as follows.

The vertex set $V$ consists of several groups of vertices, namely $c$-vertices,  $x$-vertices, $y$-vertices, $z$-vertices and path-vertices. We will introduce these gradually, one group at a time, while simultaneously adding to $E$ edges incident to vertices in the group that is being introduced, and at times also to vertices that were previously introduced. The reader can refer to Figure~\ref{fig:reduc_G} for a general representation of the graph $G$ and of the relations between its different types of vertices.

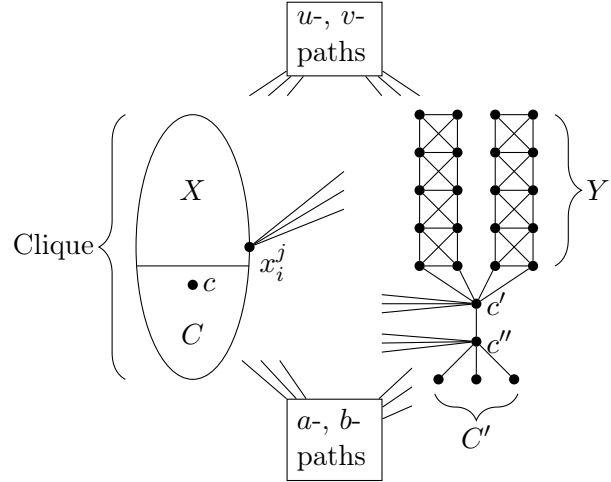
\begin{figure}[hb]
    \centering
    \begin{tikzpicture}
        \foreach \m in {1,...,5}{
            \node[v] (y\m) at (0,\m*0.5){};
            \node[v] (z\m) at (0.5,\m*0.5){};
            \draw (y\m)--(z\m);
            \node[v] (-y\m) at (1,\m*0.5){};
            \node[v] (-z\m) at (1.5,\m*0.5){};
            \draw (-y\m)--(-z\m);
        }
        \foreach \m in {1,...,4}{
            \pgfmathtruncatemacro{\n}{int(\m +1)}
            \draw(y\m)--(y\n);
            \draw(y\m)--(z\n);
            \draw(z\m)--(y\n);
            \draw(z\m)--(z\n);
            \draw(-y\m)--(-y\n);
            \draw(-y\m)--(-z\n);
            \draw(-z\m)--(-y\n);
            \draw(-z\m)--(-z\n);
        }
        \draw [decorate,decoration={brace,amplitude=10pt,raise=4pt}] (1.65,2.5)-- (1.65,0.5)node[right=12pt,midway]{$Y$};
        
        \node[v] (c') at (0.75,0){};
        \node[v] (c'') at (0.75,-0.5){};
            \node[right] at (c') {$c'$};
            \node[right] at (c'') {$c''$};
        \node[v] (c1') at (0.25,-1){};
        \node[v] (c2') at (0.75,-1){};
        \node[v] (c3') at (1.25,-1){};
        \draw (y1) -- (c') -- (z1);
        \draw (-y1) -- (c') -- (-z1);
        \draw (c') -- (c'');
        \draw (c'') -- (c1');
        \draw (c'') -- (c2');
        \draw (c'') -- (c3');
        \draw (c') -- (-0.5,0.12);        
        \draw (c') -- (-0.5,0);        
        \draw (c') -- (-0.5,-0.12);        
        \draw (c'') -- (-0.5,-0.64);
        \draw (c'') -- (-0.5,-0.52);        
        \draw (c'') -- (-0.5,-0.4);        
        
        \draw [decorate,decoration={brace,amplitude=10pt,raise=4pt}] (1.3,-1.05)-- (0.2,-1.05)node[below=12pt,midway]{$C'$};

        \draw [decorate,decoration={brace,amplitude=10pt,raise=4pt}] (-3.75,-1)-- (-3.75,2.5)node[left=12pt,midway]{Clique};
        \draw (-3,0.75) ellipse (0.75 and 1.75);
        \begin{scope}
            \clip (-3,0.75) ellipse (0.75 and 1.75);
            \draw (-5,0.5) -- (-1,0.5);
            \node at (-3,1.5){$X$};
            \node[v] (c) at(-3,0.25){};
            \node[right] at (c) {$c$};
            \node at (-3,-0.4){$C$};
        \end{scope}

        \node[v] (x) at (-2.25,0.75){};
        \node[below right = -5pt and 0 pt] at (x){$x_i^j$};
        \draw (x) -- (-1,1.25);
        \draw (x) -- (-1,1.5);
        \draw (x) -- (-1,1.75);

        \node[draw,text width=10mm] (uv) at (-1.125,3.5){$u$-, $v$-paths};
        \node[draw,text width=10mm] (ab) at (-1.125,-1.8){$a$-, $b$-paths};

        \draw (uv) -- (-0.5,2.75);
        \draw (uv) -- (-0.25,2.75);
        \draw (uv) -- (0,2.75);
        \draw (uv) -- (-1.75,2.75);
        \draw (uv) -- (-2,2.75);
        \draw (uv) -- (-2.25,2.75);

        \draw (ab) -- (-0.1,-1.35);
        \draw (ab) -- (-0.1,-1.1);
        \draw (ab) -- (-0.1,-0.85);
        \draw (ab) -- (-1.85,-0.75);
        \draw (ab) -- (-2.1,-0.75);
        \draw (ab) -- (-2.35,-0.75);
        
    \end{tikzpicture}
    \caption{General representation of $G$}
    \label{fig:reduc_G}
\end{figure}


\subsubsection*{$c$-vertices}

For every $j \in [m]$, we have a vertex $c_j$ in $V$. The set of these vertices is called $C$. 

For every $j \in [m]$, we have a vertex $c_j'$ in $V$. The set of these vertices is called $C'$. 

The vertices of $C$ and $C'$ form an anti-matching in $G$, with $\{c_j,c_j'\}$ being the anti-edges.

There are also vertices $c,c',c''\in V$ such that all three of them are connected by an edge to every vertex in $C$, $\{c',c''\}$ is an edge, and $c''$ is connected by an edge to every vertex in $C'$. 


\subsubsection*{$x$-vertices}

For every $i \in [n]$ with $i\equiv 0 \pmod{2}$, and every $j \in [m]$, we have a vertex $x_i^j$ in $V$.

For every $i \in [n]$ with $i\equiv 1 \pmod{2}$, and every $j \in [m-1]$, we have a vertex $x_i^j$ in $V$. Morevover, we also have the vertices $x_i^m,\overline{x}_i^m$.

The set of all these vertices is denoted by $X$. Morevover, $X\cup C\cup \{c\}$ is a clique in $G$. There is also an edge between each vertex of $X$ and each vertex of $C'$.

These vertices form the set of vertices in which the cop will move during the first phase of the game. Loosely, under optimal play, the cop will begin on $x_1^1$ and move from there to $x_1^2$, $x_1^3$, etc. before proceeding to the $x_2^j$, then the $x_3^j$, and so on.  The cop's precise path will correspond to a choice of values for the universally quantified variables in the \textsc{3-QBF} instance; in particular, the choice of whether to visit $x_i^m$ or $\overline{x}_i^m$ will correspond to a choice of value for $X_i$.

\subsubsection*{$y$-vertices and $z$-vertices}

For every $i \in [n]$ and every $j \in [m]$ we have vertices $y_i^j,\overline{y}_i^j,z_i^j,\overline{z}_i^j$ in $V$. We denote  
$Y_i^j:=\{y_i^j,\overline{y}_i^j,z_i^j,\overline{z}_i^j\}$, and we refer to the $Y_i^j$ as \textit{layers}; we order them in increasing lexicographical order of the pairs of their indices $(i,j)$. 
By $Y$ we denote the union of all $Y_i^j$.

For every $i \in [n]$ and $j \in [m]$, the edges $\{y_i^j,z_i^j\},\{\overline{z}_i^j,\overline{y}_i^j\}$ are in $E$.

For every $i \in [n]$ and $j \in [m-1]$, we add to $E$ the edges of the complete bipartite graph between $\{y_i^j,z_i^j\}$ and $\{y_i^{j+1}, z_i^{j+1}\}$, and also between $\{\overline{y}_i^j,\overline{z}_i^j\}$ and $\{\overline{y}_i^{j+1},\overline{z}_i^{j+1}\}$. For every $i \in [n-1]$, we also add to $E$ the edges of the complete bipartite graph between $Y_i^m$ and $Y_{i+1}^1$ (see Figure~\ref{fig:reduc_Y} for a representation of some $y$-vertices and $z$-vertices, and the edges between them).

\begin{figure}[hb]
    \centering
    \begin{tikzpicture}[scale = 1.5]
        \node[v] (y11) at (0,0){};
        \node[v] (y12) at (0,-0.5){};
        \node[v] (y13) at (0,-1){};
        
        \node[v] (z11) at (0.5,0){};
        \node[v] (z12) at (0.5,-0.5){};
        \node[v] (z13) at (0.5,-1){};
        
        \node[v] (-y11) at (1.5,0){};
        \node[v] (-y12) at (1.5,-0.5){};
        \node[v] (-y13) at (1.5,-1){};
        
        \node[v] (-z11) at (2,0){};
        \node[v] (-z12) at (2,-0.5){};
        \node[v] (-z13) at (2,-1){};
        
        \node[v] (y21) at (0,-1.75){};
        \node[v] (y22) at (0,-2.25){};
        \node[v] (y23) at (0,-2.75){};
        
        \node[v] (z21) at (0.5,-1.75){};
        \node[v] (z22) at (0.5,-2.25){};
        \node[v] (z23) at (0.5,-2.75){};
        
        \node[v] (-y21) at (1.5,-1.75){};
        \node[v] (-y22) at (1.5,-2.25){};
        \node[v] (-y23) at (1.5,-2.75){};
        
        \node[v] (-z21) at (2,-1.75){};
        \node[v] (-z22) at (2,-2.25){};
        \node[v] (-z23) at (2,-2.75){};

        \node[above left = -5pt and 0pt] at (y11) {$y_1^1$};
        \node[above left = -5pt and 0pt] at (y12){$y_1^2$};
        \node[above left = -5pt and 0pt] at (y13){$y_1^3$};
        \node[below left = -5pt and 0pt] at (y21){$y_2^1$};
        \node[below left = -5pt and 0pt] at (y22){$y_2^2$};
        \node[below left = -5pt and 0pt] at (y23){$y_2^3$};
        \node[above left = -5pt and 0pt] at (-y11){$\overline{y}_1^1$};
        \node[above left = -5pt and 0pt] at (-y12){$\overline{y}_1^2$};
        \node[above left = -5pt and 0pt] at (-y13){$\overline{y}_1^3$};
        \node[below left = -5pt and 0pt] at (-y21){$\overline{y}_2^1$};
        \node[below left = -5pt and 0pt] at (-y22){$\overline{y}_2^2$};
        \node[below left = -5pt and 0pt] at (-y23){$\overline{y}_2^3$};
        
        \node[above right = -5pt and 0pt] at (z11){$z_1^1$};
        \node[above right = -5pt and 0pt] at (z12){$z_1^2$};
        \node[above right = -5pt and 0pt] at (z13){$z_1^3$};
        \node[below right = -5pt and 0pt] at (z21){$z_2^1$};
        \node[below right = -5pt and 0pt] at (z22){$z_2^2$};
        \node[below right = -5pt and 0pt] at (z23){$z_2^3$};
        \node[above right = -5pt and 0pt] at (-z11){$\overline{z}_1^1$};
        \node[above right = -5pt and 0pt] at (-z12){$\overline{z}_1^2$};
        \node[above right = -5pt and 0pt] at (-z13){$\overline{z}_1^3$};
        \node[below right = -5pt and 0pt] at (-z21){$\overline{z}_2^1$};
        \node[below right = -5pt and 0pt] at (-z22){$\overline{z}_2^2$};
        \node[below right = -5pt and 0pt] at (-z23){$\overline{z}_2^3$};

        \draw (y11)--(z11);
        \draw (y12)--(z12);
        \draw (y13)--(z13);
        \draw (-y11)--(-z11);
        \draw (-y12)--(-z12);
        \draw (-y13)--(-z13);
        \draw (y21)--(z21);
        \draw (y22)--(z22);
        \draw (y23)--(z23);
        \draw (-y21)--(-z21);
        \draw (-y22)--(-z22);
        \draw (-y23)--(-z23);

        \draw (y11)--(y12);
        \draw (y11)--(z12);
        \draw (z11)--(y12);
        \draw (z11)--(z12);
        \draw (y12)--(y13);
        \draw (y12)--(z13);
        \draw (z12)--(y13);
        \draw (z12)--(z13);
        
        \draw (y21)--(y22);
        \draw (y21)--(z22);
        \draw (z21)--(y22);
        \draw (z21)--(z22);
        \draw (y22)--(y23);
        \draw (y22)--(z23);
        \draw (z22)--(y23);
        \draw (z22)--(z23);
        
        \draw (-y11)--(-y12);
        \draw (-y11)--(-z12);
        \draw (-z11)--(-y12);
        \draw (-z11)--(-z12);
        \draw (-y12)--(-y13);
        \draw (-y12)--(-z13);
        \draw (-z12)--(-y13);
        \draw (-z12)--(-z13);
        
        \draw (-y21)--(-y22);
        \draw (-y21)--(-z22);
        \draw (-z21)--(-y22);
        \draw (-z21)--(-z22);
        \draw (-y22)--(-y23);
        \draw (-y22)--(-z23);
        \draw (-z22)--(-y23);
        \draw (-z22)--(-z23);

        \draw (y13)--(y21);
        \draw (y13)--(z21);
        \draw (y13)--(-y21);
        \draw (y13)--(-z21);
        
        \draw (z13)--(y21);
        \draw (z13)--(z21);
        \draw (z13)--(-y21);
        \draw (z13)--(-z21);
        
        \draw (-y13)--(y21);
        \draw (-y13)--(z21);
        \draw (-y13)--(-y21);
        \draw (-y13)--(-z21);
        
        \draw (-z13)--(y21);
        \draw (-z13)--(z21);
        \draw (-z13)--(-y21);
        \draw (-z13)--(-z21);
        
    \end{tikzpicture}
    \caption{$y$-vertices and $z$-vertices if $n=2$ and $m=3$}
    \label{fig:reduc_Y}
\end{figure}
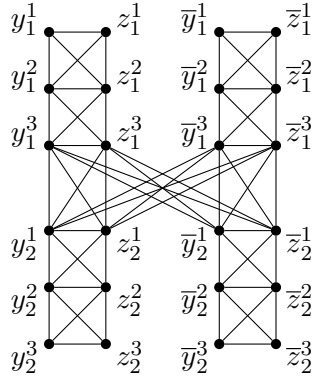

Furthermore, $c$ is connected by an edge to every vertex of $Y$ and $c'$ is connected by an edge to every vertex of $Y_n^m$. For all $i,i'\in[n]^2$ and $j,j' \in [m]^2$, if $(i,j)$ is larger than $(i',j')$ under the lexicographic ordering, then $x_i^j$ and $\overline{x}_i^j$ (if it exists) are connected by an edge to every vertex of $Y_{i'}^{j'}$ .

The $y$-vertices and $z$-vertices together form the set of vertices on which the robber will move during the first phase of the game. Morevover, they also correspond assignments for the existentially-quantified variables in the \textsc{3-QBF} instance. If the robber goes to the vertices $\{y_i^j,z_i^j\}$ it corresponds to the variable $X_i$ being set to true, and if the robber goes to the vertices $\{\overline{y}_i^j,\overline{z}_i^j\}$ it corresponds to the variable $X_i$ being set to false.

If $X_i$ is a literal appearing in the clause $C_j$ of the \textsc{3-QBF} instance, then $\overline{y}_i^j$ is connected by an edge to $c_j$, and $\overline{z}_i^j$ is connected by an edge to $c_j'$. If, however, $\overline{X_i}$ is a literal in $C_j$, then there is an edge between $y_i^j$ and $c_j$, and there is an edge between $z_i^j$ and $c_j'$.

If $i \equiv 1 \pmod{2}$ and $i< n$, then there are edges from $x_i^m$ to both $\overline{y}_{i+1}^1$ and $\overline{z}_{i+1}^1$, and there are edges from $\overline{x}_i^m$ to both $y_{i+1}^1$ and $z_{i+1}^1$.

The vertices of $X$ are connected to all vertices of $Y$ with a lower lexicographical value of $(i,j)$; this prevents the robber from going back to a vertex with a smaller pair of indices as long as the cop stays in $X$.

\subsubsection*{Path vertices}

For every $k \in [N]$, we have a vertex $u^0_k$ in $V$.

For every $i \in [n]$, $j \in [m]$, $k \in [N-1]$, we have vertices $u_k^{i,j}$ and $v_k^{i,j}$ in $V$. 
For every $k \in [N-1]$, we have vertices $a_k'$, $b_k'$, $a_k''$ and $b_k''$ in $V$, and for every $j\in [m]$, $k\in [N-1]$, we have vertices $a^j_k$ and $b^j_k$ in $V$.

Each of the sequences of vertices $(u^0_k)_k$, $(u_k^{i,j})_k$, $(v_k^{i,j})_k$, $(a_k')_k$, $(b_k')_k$, $(a_k'')_k$, $(b_k'')_k$, $(a^j_k)_k$ and $(b^j_k)_k$ induces a path in $G$, with the edges connecting vertices whose $k$-index differs by one.

Moreover, there is an edge between $u_1^0$ and $x_1^1$, and both $a_1'$ and $b_1'$ are connected by an edge to both $c$ and $c'$. Both $a_1''$ and $b_1''$ are connected by an edge to $c''$, and they are connected by an edge to every vertex of $C$ and every vertex of $C'$.

For $i \in [n]$ and $j \in [m]$, both $u_1^{i,j}$ and $v_1^{i,j}$ are connected to every vertex of $Y_i^j$, to $x_i^j$, and to $\overline{x}_i^j$ if it exists (see Figure~\ref{fig:reduc_path} for a representation of a $u$-path and a $v$-path and their edges to $x$-vertices, $y$-vertices and $z$ vertices).

For $j \in [m]$, $a^j_1$ and $b^j_1$ are connected to $c_j$ and $c_j'$.

All these path-vertices are used to force the moves of the cop -- if the robber manages to get to the starting vertex of one of these paths without getting caught, then he can traverse it and damage enough vertices to win the game. Therefore, the cop needs to guard the starting vertices of the paths to prevent that from happening.

\begin{figure}
    \centering
    \begin{tikzpicture}[scale = 1.5]
        \node[v] (y) at (0,0){};
        \node[v] (z) at (0.5,0){};
        \node[v] (-y) at (1,0){};
        \node[v] (-z) at (1.5,0){};
        \node[v] (x) at (-0.5,0){};
        
        \node[v] (u1) at (0.25,-0.75){};
        \node[v] (u2) at (0.25,-1.25){};
        \node[v] (uN) at (0.25,-2){};
        
        \node[v] (v1) at (0.75,-0.75){};
        \node[v] (v2) at (0.75,-1.25){};
        \node[v] (vN) at (0.75,-2){};
        
        \draw (y)--(z);
        \draw (-y)--(-z);
        \draw (u1)--(x)--(v1);
        \draw (u1)--(y)--(v1);
        \draw (u1)--(z)--(v1);
        \draw (u1)--(-y)--(v1);
        \draw (u1)--(-z)--(v1);
        \draw (u1)--(u2)--(0.25,-1.4);
        \node at (0.25,-1.525){.};
        \node at (0.25,-1.625){.};
        \node at (0.25,-1.725){.};
        \draw (uN)--(0.25,-1.85);
        \draw (v1)--(v2)--(0.75,-1.4);
        \node at (0.75,-1.525){.};
        \node at (0.75,-1.625){.};
        \node at (0.75,-1.725){.};
        \draw (vN)--(0.75,-1.85);

        \node[above] at (x) {$x_i^j$};
        \node[above] at (y) {$y_i^j$};
        \node[above] at (z) {$z_i^j$};
        \node[above] at (-y) {$\overline{y}_i^j$};
        \node[above] at (-z) {$\overline{z}_i^j$};

        \node[left] at (u1){$u_1^{i,j}$};
        \node[left] at (u2){$u_2^{i,j}$};
        \node[left] at (uN){$u_{N-1}^{i,j}$};
        
        \node[right] at (v1){$v_1^{i,j}$};
        \node[right] at (v2){$v_2^{i,j}$};
        \node[right] at (vN){$v_{N-1}^{i,j}$};
        
    \end{tikzpicture}
    \caption{Path vertices and their connections for some arbitrary $i$, $j$}
    \label{fig:reduc_path}
\end{figure}
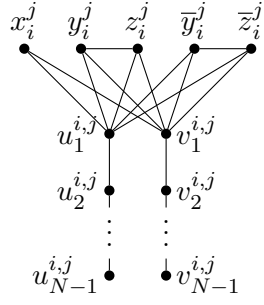

\subsection{Regular play}

We now analyze the game on $G$.  In what follows, we say that the robber ``wins'' the game if he is able to damage at least $N$ vertices; otherwise, the cop wins.

In this subsection we describe \textit{regular play}, i.e. how the game proceeds when both players play ``as expected''. Initially, the cop starts at $x_1^1$, and the robber starts in $Y_1^1$. Then, the cop stays in their position and the robber goes to the adjacent vertex in $Y_1^1$. After that, the robber goes to the following layer.

Each time the robber reaches a new layer $Y_i^j$, the cop responds by going to $x_i^j$ or $\overline{x}_i^j$ (if it exists).
As their following move, the robber goes to the adjacent vertex in the same layer, while the cop does not move. This completes the \textit{phase} corresponding to layer $Y_i^j$.

Once a phase is complete, the robber proceeds to the following layer, another phase starts, and this proceeds all the way to the last layer.

After the phase on the last layer is complete, the robber goes to $c'$, and the cop responds by going to $c$. The robber then goes to $c''$, the cop goes to a vertex $c_j$ of $C$, and the robber proceeds to $c'_j$. 
The cop stays put for a round, while the robber goes to a vertex $z_i^j$.
Then the cop goes to $y_i^j$.

At this point we consider regular play to be over.

\begin{lem}\label{lem:regular_play}
    If both the cop and the robber follow regular play, then the robber has a winning strategy on $G$ if and only if Satisfier has a winning strategy for the underlying \textsc{3-QBF} instance. Moreover, if the robber cannot win within regular play, then he also cannot win by continuing to play after the end of regular play.
\end{lem}

\begin{proof}

    Note that the only moments where players have choices during regular play are the following: 
    \begin{itemize}[noitemsep,topsep=5pt,parsep=5pt,partopsep=5pt]
        \item in the first move, the robber chooses between $y^1_1$ and $\overline{y}^1_1$ (note that starting on $z_1^1$ or $\overline{z}_1^1$ would be equivalent since we are under normal play and the cop will pass their next move);
        \item for $i$ even, the robber chooses between moving from the layer $Y_{i}^m$ to either $y_{i+1}^1$ or $\overline{y}_{i+1}^1$ (note that going to $z_{i+1}^1$ or $\overline{z}_{i+1}^1$ would be equivalent);
        \item for $i$ even, the cop chooses between moving from $x_i^{m-1}$ to either $x_i^m$ or $\overline{x}_i^m$;
        \item when moving to $c$-vertices, the cop chooses which $c_j$ to move to, for some $j\in [m]$; and
        \item in their last move, the robber chooses some vertex $z_i^j$ or $\overline{z}_i^j$ to move to. 
    \end{itemize}
    
    Since these are the only situations where a choice is made, these are the only moves that we will consider in the following. We note, however, that whenever the robber enters a vertex $y_i^1$, they damage all the vertices of the form $y_i^j$ and $z_i^j$ and leave all the vertices of the form $\overline{y}_i^j$ and $\overline{z}_i^j$ undamaged, and the opposite happens if the robber enters the vertex $\overline{y}_i^j$. 

    Let us first show that if Satisfier has a winning strategy $\mathcal S$ in the \textsc{3-QBF} game, then the robber has a winning strategy in the game on $G$ if both players follow regular play. As Satisfier plays first in \textsc{3-QBF}, they are the player choosing the value of $X_1$. If, according to $\mathcal S$, Satisfier chooses $X_1$ to be true, then the robber starts on $y_1^1$, otherwise they start on $\overline{y}_1^1$. During the rest of the game, whenever the cop moves to the vertex $x_i^m$, the robber will have to move to $y_{i+1}^1$ and he will consider that Falsifier set the variable $X_{i+1}$ to true in the \textsc{3-QBF} game, and whenever the cop moves to vertex $\overline{x}_i^m$, the robber will have to move to $\overline{y}_{i+1}^1$ and will consider that Falsifier set the variable $X_{i+1}$ to false. When the robber has to make a choice between moving to the vertex $y_i^1$ or the vertex $\overline{y}_i^1$, they observe the value that Satisfier sets the variable $X_i$ to, according to $\mathcal S$, and they move to $y_i^1$ if it is set to true, and to $\overline{y}_i^1$ otherwise.
    
    Note that after the robber has finished moving through the $y$-vertices, if some variable $X_i$ is set to true then the vertices $\overline{y}_i^j$ and $\overline{z}_i^j$ for all $j$ will remain undamaged, and if $X_i$ is set to false then the vertices $y_i^j$ and $z_i^j$ for all $j$ will remain undammaged.

    Let $c_j$ be the vertex that the cop chose to move to towards the end of the game, when stepping onto the $c$-vertices.
    Since $\mathcal S$ is a winning strategy for Satisifer, the clause $C_j$ is satisified in the 3-QBF game. Let $X_i$ be a variable that satisfies $C_j$. Recall that if $X_i$ appears positively in $C_j$ then $c_j'$ is connected to $\overline{z}_i^j$ in $G$, and if it appears negatively then $c_j'$ is connected to $z_i^j$. If $X_i$ appears positively in $C_j$ then it is set to true and $c_j'$ is connected to $\overline{z}_i^j$, which was previously undamaged, and the last move of the robber is to move to this vertex. If $X_i$ appears negatively then it is set to false and the robber can move from $c_j'$ to the undamaged vertex $z_i^j$.

    It remains to count how many vertices the robber damages during this game. For each $i \in [n]$, they either damage the $y_i^j$ and $z_i^j$ vertices for all $j$, or the $\overline{y}_i^j$ and $\overline{z}_i^j$ vertices, for all $j$. That is $2nm$ vertices altogether, when summed up for all $i$. The robber also damages $c'$, $c''$ and $c'_j$. Lastly, since $z_i^j$ or $\overline{z}_i^j$ was previously undamaged, and since in their following move the cop does not capture the robber, this vertex is also damaged. In total, the robber damages $N=2nm+4$ vertices and wins the game. \medskip

    We proceeed to proving the opposite direction, showing that if Falsifier has a winning strategy $\mathcal S$ in the \textsc{3-QBF} game, then the cop has a strategy to prevent the robber from damaging $N$ vertices in $G$, if both players follow regular play. For $i$ odd, including $i=1$, whenever the robber chooses to move to some vertex $y_i^1$, the cop considers that Satisfier set the variable $X_i$ to true, and if the robber moves to $\overline{y}_i^1$, the cop considers that Satisfier set the variable $X_i$ to false. When the cop has to choose between moving to $x_i^m$ or $\overline{x}_i^m$, they consider what value the strategy $\mathcal S$ assigns to $X_{i+1}$. If true is assigned to $X_{i+1}$ then the cop moves to $x_i^j$, forcing the robber to move to $y_i^j$, and otherwise the cop moves to $\overline{x}_i^j$, forcing the robber to move to $\overline{y}_i^j$.

    After the robber has finished moving through the $y$-vertices, for every $i$, if $X_i$ is set to true both $y_i^j$ and $z_i^j$ will be damaged for all $j$, and if $X_i$ is set to false both $\overline{y}_i^j$ and $\overline{z}_i^j$ will be damaged for all $j$.

    Since $\mathcal S$ is a winning strategy for Falsifier, there exists an unsatisfied clause $C_j$ and the cop will move to the vertex $c_j$, thus forcing the robber to move to $c_j'$. Since the clause $C_j$ is unsatisfied, all the $y$-vertices and $z$-vertices accessible from $c_j'$ are already damaged. Let $i$ be such that the robber moves to $z_i^j$ or $\overline{z}_i^j$. The cop concludes regular play by moving to $y_i^j$ or $\overline{y}_i^j$.

    Similarly to the previous case, the robber damaged $2mn$ vertices while going through $Y$ in the first phase of the game, and damaged three more vertices $c'$, $c''$ and $c_j'$. However, since the last move of the robber was to a vertex that was already damaged, only $N-1$ vertices are damaged at this point.

    All that remains to be shown is that no matter what the robber does after regular play is over, they cannot damage any other vertices. At the end of regular play, the only vertex that is a neighbor of the robber but not of the cop is $c_j'$ and therefore the robber has to move to this vertex if they do not want to be captured in the following move of the cop. The cop can then just go back to $c_j$. At this point all the vertices that are neighbors of $c_j'$ but not of $c_j$ are the vertices $z_i^j$ and $\overline{z}_i^j$ corresponding to variables $X_i$ appearing in $C_j$. However, as explained before, all these vertices are already damaged and if the robber goes to one of them, the cop goes to the corresponding $y$-vertex, recreating the position at the end of regular play. Therefore, the robber cannot damage other vertices and the cop wins the game. 
\end{proof}

\subsection{The cop does not follow regular play}

Here we explore the various ways the cop can deviate from regular play, and we show that doing so is not to their advantage.

\begin{lem}\label{lem:cop_regular}
    If the cop has a winning strategy on $(G,N)$, then they have a winning strategy where they follow regular play.
\end{lem}

\begin{proof}

Let us assume that the cop has a winning strategy on $G$.

Suppose the cop does not start at $x_1^1$. If they start somewhere on the $(u^0_k)_k$ path, then the robber starts at $y_1^1$ and the cop will not be able to stop the robber from safely moving to $u_1^1$, and subsequently damaging $y_1^1$ and all the $N-1$ vertices on the path $u_1^1, \dots, u_{N-1}^1$. If the cop starts somewhere else, then the robber can start at $u_1^0$ and traverse the whole $(u^0_k)_k$ path, thus winning the game. Therefore, we can assume that the cop follows regular play on the first move.

Suppose that the robber is within $Y_i^j$, that it is the cop's turn to play and that both players have followed regular play until this point. Since the cop followed regular play until now, he is either on $x_i^j$, on $\overline{x}_i^j$ (if it exists), or on one of the $x$-vertices immediately preceding these in the regular play order. If the cop does not move to, or stay at, either $x_i^j$ or $\overline{x}_i^j$, then the only vertices that they can reach are: the $x$-vertices, the $c$-vertices, and either $u_1^{i,j}$ and $v_1^{i,j}$, or the $u_1$-vertex and $v_1$-vertices immediately preceding them. If the cop moves to any of them except $u_1^{i,j}$, then the robber can move safely to $u_1^{i,j}$, they will be able to safely damage the $N-1$ vertices of its path, and they will win, because they also damaged one vertex in $Y_i^j$. If the cop moves to $u_1^{i,j}$, then the robber can move to $v_1^{i,j}$ and the situation is similar. Therefore, we can assume that the cop follows regular play until the robber moves to $c'$.

Similarly, if the cop does not follow regular play when the robber is in $\{c',c''\}\cup C'$, then the robber would be able to safely move to the start of an $a$-path or a $b$-path, thus winning the game. Therefore, we can assume that the cop follows regular play until the last move.

On the robber's last move of regular play, since we assumed that the cop has a winning stategy, the robber must not be able to damage another vertex. In this case, we showed in the proof of Lemma~\ref{lem:regular_play} that for the cop, moving to the $y$-vertex opposite of the $z$-vertex occupied by the robber would prevent the robber from damaging more vertices. This completes the proof that if the cop has a winning strategy on $G$, then they have one where they follow regular play.
\end{proof}

\subsection{The robber does not follow regular play}

Here we explore the various ways the robber can deviate from regular play, and show that doing so is not to their advantage.

\begin{lem}\label{lem:robber_regular}
    If the robber has a winning strategy on $(G,N)$, then they have a winning strategy where they follow regular play.
\end{lem}

\begin{proof}

Let us assume that the robber has a winning strategy on $G$.

As argued in Lemma \ref{lem:cop_regular}, the cop starts the game on $x^1_1$ (as prescribed by regular play). Suppose that the robber does not start on a vertex of $Y_1^1$. If they go to any vertex of $X\cup C\cup C'\cup \{c, u_1^0, u_1^{1,1}, v_1^{1,1}\}$, then the cop can capture them immediately, thus winning. If the robber starts on a vertex $u_k^0$ with $k\geq 2$, then the cop can move to $u_1^0$ on his first turn and the robber will only be able to damage at most $N-1$ vertices in this path, thus losing.
Similarly, if the robber starts on any vertex of a $u$-, $v$-, $a$- or $b$-path that is not the starting vertex of the path, then the cop will move to a vertex of $X\cup C \cup \{c\}$ adjacent to the starting vertex of the path, trapping the robber inside the path and keeping them from damaging more than $N-2$ vertices.

If the robber starts on $u_1^{i,j}$ or $v_1^{i,j}$, with $(i,j) \neq (1,1)$, then the cop will move to $x_i^j$. If the robber goes further into the path, then the cop will move to the start of the path, trapping the robber inside; consequently, he will damage at most $N-1$ vertices. If the robber moves to a vertex of $Y_i^j$, then he did not damage any vertex in the previous $Y$-layers and damaged $2k-1$ fewer vertices than they would have if they reached their position following regular play, where $k$ is the number of skipped layers (the $-1$ comes from the fact that they would not have damaged the starting vertex of the path under regular play). Then the cop can continue the game as if the position occurred in regular play. The only situation that in which this is not strictly determinental for the robber is when they only skipped one layer and then managed to damage the skipped $z$/$\overline{z}$-vertex at the very end of the game. In this situation, the robber would damage the same number of vertices by starting in $Y_1^1$. 

The situation is even worse for the robber if their starting vertex is the first vertex of an $a$- or $b$-path. The cop can move to a vertex of $C\cup\{c\}$ guarding the path. If the robber goes further in the path as before, then he only damages $N-1$ vertices. If instead he moves to an adjacent vertex of $C'\cup\{c',c''\}$, then the cop can move so as to create a configuration that could arise under regular play; the robber did not damage any vertices of $Y$ and has thus damaged $2nm-1$ vertices fewer than he would have under regular play. 

Lastly, if the robber starts on a vertex of $Y_i^j$ with $(i,j) \neq (1,1)$ or a vertex of $\{c',c''\}$ then, after the cop moves to an appropriate vertex of $X\cup C\cup \{c\}$ guarding their adjacent paths, the situation is similar to the previous situation where they started on the first vertex of a path and then moved to one of these vertices, except that the has robber damaged one less vertex.

Therefore, we can assume that the robber starts on a vertex of $Y_1^1$.

Suppose that it is the robber's turn and that they are in some layer $Y_i^j$. Since the cop follows regular play, they are on $x_i^j$ or $\overline{x}_i^j$. In particular, if the robber moves to the vertices $u_i^j$ and $v_i^j$, one vertex of $C'$ or a vertex of the previous $Y$-layer, the cop can capture him immediately.

Therefore, the only vertices that the robber can safely reach are the vertices on the next $Y$-layer (or $c'$ if the robber is in the $Y_n^m$ layer) and the $y$-, $z$-, $\overline{y}$- or $\overline{z}$-vertex of $Y_i^j$ adjacent to their current vertex. If this last vertex was not already damaged then it is optimal to move to it (as prescribed by regular play) because the cop will have to stay on the same vertex to guard the vertices $u_i^j$ and $v_i^j$. If it was already damaged, then the robber has to move to the next layer or to $c'$; both of these types of moves follow regular play. 

Suppose that it is the robber's turn and that they are on $c'$. Since the cop is on $c$ according to regular play, the only vertex that the robber can safely move to is $c''$. Similarly, if the robber is on $c''$ then the cop is on a vertex $c_j$ and the only safe move of the robber is to move to $c_j'$ and after this move the only safe moves are to move to some $z$- and $\overline{z}$-vertices. All of these moves follow regular play. As we have seen in Lemma~\ref{lem:regular_play}, if all the neighbors of $c_j$ are already damaged, then the cop can win the game so we can assume that at least one of them is undamaged and the robber can move to it as the last move of regular play and win the game.
\end{proof}

We can finally use all the previous lemmas to prove Theorem~\ref{complexity}.

\begingroup
\renewcommand{\thetheorem}{\ref{complexity}} 
\begin{theorem}
The decision problem \textsc{Damage Number} is {\sf PSPACE}-complete.
\end{theorem}
\addtocounter{theorem}{-1} 
\endgroup

\begin{proof}
As previously stated, \textsc{Damage Number} is in {\sf PSPACE} as there are only a polynomial number of configurations for the cop and the robber and we can assume that between repetitions of the same configuration, at least one more vertex has been damaged. 

For the hardness proof, consider $\varphi(X_1,...,X_n)$ a 3-CNF formula with $m$ clauses. Let $(G,N)$ be the pair obtained from $\varphi$ by following the construction of section~\ref{sub:reduc}. This construction has polynomial size as both $N$ and $G$ are polynomial in $\varphi$.

According to Lemma~\ref{lem:regular_play} when both players follow regular play, if Satisfier (Falsifier resp.) has a winning strategy in the 3-QBF game played on $\varphi$, the robber (the cop resp.) has a winning strategy on $(G,N)$.

Moreover, according to Lemma~\ref{lem:robber_regular} and Lemma~\ref{lem:cop_regular}, the robber has a winning strategy on $(G,N)$ if and only if Satisfier has a winning strategy on $\varphi$. Therefore, the problem \textsc{Damage Number} is also {\sf PSPACE}-hard, thus concluding the proof.
\end{proof}

\section{Concluding remarks} \label{s:conclusion}

To the best of our knowledge, prior to this work, no sequences of $n$-vertex graphs were known with damage number much larger than $n/2$. One can address the natural extremal problem of determining $\dmg(n) := \max\{\dmg(G):|V(G)|=n\}$. Theorem \ref{thm:proj_plane2} on projective planes shows that for an infinite number of values of $n$, we have $\dmg(n)\ge n-O(\sqrt{n})$. On the other hand, it is not hard to see that $\dmg(G_n)\le n-\diam(G_n)/2$: the cop can always protect at least $\diam(G_n)/2$ vertices by starting the game in the center of a diametric path, and guarding it using the strategy of Aigner and Fromme~\cite{AignerFromme} for guarding an isometric path. This and the bound $\dmg(G_n)\le n-\Delta(G)-1$ from \cite{CS} imply $\dmg(n)\le n-\Omega(\frac{\log n}{\log\log n})$. It would be interesting to see results closing the gap between the upper and lower bounds. If the result on the Erd\H os-R\'enyi random graph could be extended to values to $p=o(n^{-1/2})$, then it would improve the lower bound for $\dmg(n)$.

Proving that $\dmg(n) \le n - f(n)$ for some ``small'' function $f$ would lead to bounds on the cop number of a connected graph.  For example, if we knew that the damage number of an $n$-vertex graph was bounded above by $n-n^{\eps}$ for some positive constant $\eps$ smaller than 1, then it would follow that the cop number of a connected $n$-vertex graph is bounded above by $k \cdot n^{1-\eps}$ for some constant $k$, thereby resolving the ``soft'' Meyniel conjecture (which asserts that every $n$-vertex connected graph has cop number $O(n^{1-\eps})$ for some $\eps > 0$) ~\cite{BB}.  To see this, note that the inequality $\dmg(n) \le n-n^{\eps}$ would imply that a single cop, when playing the game on a graph $G$, can eventually restrict the robber to a connected set $S$ of at most $n-n^{\eps}$ vertices.  Let $G'$ denote the subgraph of $G$ induced by $S$; we have $\size{V(G')} \le n-n^{\eps}$, so by induction, $k(n-n^{\eps})^{1-\eps}$ additional cops would suffice to capture a robber on $G'$ which, together with the original cop, would lead to a capture on $G$.  Consequently, letting $c(G)$ denote the cop number of $G$, we have
\begin{align*}
c(G) \le c(G') + 1 &\le k(n-n^{\eps})^{1-\eps} + 1\\ 
     &= kn^{1-\eps}(1-n^{\eps-1})^{1-\eps} + 1\\ 
     &\le kn^{1-\eps}(1-(1-\eps)n^{\eps-1}) + 1\\ 
     &= k(n^{1-\eps} - (1-\eps)) + 1\\ 
     &= kn^{1-\eps} - (1-\eps)k + 1,
\end{align*}
which is at most $kn^{1-\eps}$ for $k$ sufficiently large relative to $\eps$.  Since there exist graphs with cop number $O(\sqrt{n})$, we cannot hope to have $f(n) \gg n^{1/2}$; that is, there must exist graphs with damage number at least $n-k\cdot n^{1/2}$ for some $k$, and indeed, Theorem \ref{thm:proj_plane2} gives an explicit sequence of such graphs.

\medskip

    We studied graphs related to set families. Another natural candidate would be the \textit{Kneser graph} $\KG(n,k)$ that has vertex set $\binom{[n]}{k}$ and two sets $F,G$ are joined by an edge if and only if $F\cap G=\emptyset$. This is empty if $n<2k$ and is a matching if $n=2k$, so the first meaning ful value is $n=2k+1$. An argument based on Lemma~\ref{lem:generallower}, would yield $\dmg(\KG(2k+1,k))\ge \binom{2k-1}{k-1}$. Also, for the upper bound the cop can use an auxiliary game on the middle graph $M_k$ as follows: if the robber / cop is at vertex representing set $F$ in the game in $\KG(2k+1,k)$ after an even number of rounds, then the robber / cop is at vertex $F$ in the game on $M_k$. On the other hand, if  the robber / cop is at vertex representing set $F$ in the game in $\KG(2k+1,k)$ after an odd number of rounds, then the robber / cop is at vertex $[n]\setminus F$ in the game on $M_k$. A strategy of the cop in the $M_k$-game yields a strategy in the $\KG(2k+1,k)$-game. Also, it is not hard to see that an optimal strategy of the cop in the $M_k$-game converted to the $\KG(2k+1,k)$-game shows $\dmg(\KG(2k+1,k))\le 2\binom{2k-1}{k-1}$. We could not close the gap between the lower and upper bounds. Also, as $n$ grows the graph $\KG(n,k)$ behaves very differently, for example its diameter becomes 2 as soon as $n\ge 3k-1$. Determining $\dmg(\KG(n,k))$ is a natural open problem.

\section*{Acknowledgments} 

This work was initiated at the \emph{4th Workshop on Games on Graphs} on Rogla, Slovenia in June 2026. We thank the organizers and all the participants for the inspiring atmosphere.


\begin{thebibliography}{99}
    \bibitem{AignerFromme}
    M.~Aigner, M.~Fromme, A game of cops and robbers, Discrete~Appl.~Math.~8 (1984), no. 1, 1--11.
    \bibitem{AS}
    N.~Alon, J.~Spencer, \emph{The probabilistic method}, 4th ed., John Wiley \& Sons, 2016.
    \bibitem{BB}
    W.~Baird, A.~Bonato, Meyniel’s conjecture on the cop number: A survey, Journal of Combinatorics, 3(2) (2012), 225--238.
    \bibitem{bonato2011game} 
    A.~Bonato, R.~Nowakowski, \emph{The Game of Cops and Robbers on Graphs}, American Mathematical Society, Providence, Rhode Island, 2011.
    \bibitem{bonato-pralat-book}
    A.~Bonato, P.~Pra{\l}at, \emph{Graph searching games and probabilistic methods}, Chapman and Hall/CRC, 2017.
    \bibitem{carlson2021damage}
    J.~Carlson, R.~Eagleton, J.~Geneson, J.~Petrucci, C.~Reinhart, P.~Sen, The damage throttling number of a graph, Australas.~J.~Comb.~80 (2021), 361--385.
    \bibitem{carlson2022multi}
    J.~Carlson, M.~Halloran, C.~Reinhart, The multi-robber damage number of a graph, \emph{arXiv preprint arXiv:2205.06956}, 2022. \url{https://arxiv.org/abs/2205.06956}
    \bibitem{CS}
    D.~Cox, A.~Sanaei, The damage number of a graph, Australas.~J.~Comb.~75 (2019), 1--16.
    \bibitem{gagarova2025exploring}
    S.~Gagarova, Exploring the Multi-Robber Damage Number of a Graph, SIAM Undergraduate Research Online (SIURO) 19 (2026), 97--117.
    \bibitem{GareyJohnson}
    M.R.~Garey, D.S.~Johnson, Computers and Intractability: A Guide to the Theory of NP-Completeness. 
    W.H.~Freeman and Co., San Francisco, CA, 1979. 
    \bibitem{HMP}
    M.A.~Huggan, M.E.~Messinger, A.~Porter, The damage number of the Cartesian product of graphs, Australas.~J.~Comb.~88 (2024), 362--384.
    \bibitem{cor-huggan2024damage}
    M.A.~Huggan, M.E.~Messinger, A.~Porter, Corrigendum to: The damage number of the Cartesian product of graphs, Australas.~J.~Comb.~91 (2025), 217--218.
    \bibitem{KissSzonyi}
    Gy.~Kiss, T.~Sz\H onyi, \textit{Finite geometries}, Chapman and Hall/CRC, 2019.
    \bibitem{nowakowski1983vertex}
    R.~Nowakowski,  P.~Winkler, Vertex-to-vertex pursuit in a graph, Discrete Mathematics~43, 235--239, 1983.
    \bibitem{stojakovic2022multirobber}
    M.~Stojakovi\'c, L.~Wulf, On the multi-robber damage number, \emph{arXiv preprint arXiv:2209.10965}, 2022. \url{https://arxiv.org/abs/2209.10965}.
    \bibitem{quilliot1978jeux}
    A.~Quilliot, \emph{Jeux et pointes fixes sur les graphes}, Ph.D. dissertation, Université de Paris VI, 1978.
\end{thebibliography}
\end{document}